\documentclass[11pt]{amsart}
\usepackage{appendix}
\usepackage{geometry}
\usepackage{amssymb}
\usepackage{amsmath}
\usepackage{amsthm}
\usepackage{mathtools}
\usepackage{mathrsfs}
\usepackage{tikz}
\usepackage{svg}
\usepackage{tikz-cd}
\DeclareFontFamily{U}{rsfs}{\skewchar\font127 }
\DeclareFontShape{U}{rsfs}{m}{n}{%
   <-6> rsfs5
   <6-8> rsfs7
   <8-> rsfs10
}{}
\usepackage{bm}
\usepackage{hyperref}
\hypersetup{
	colorlinks=true,
	linkcolor=cyan!50!blue,
	filecolor=blue,      
	urlcolor=red,
	citecolor=blue,
}
\usepackage{stackengine}
\stackMath

\usepackage{cleveref}
\usepackage{makecell}
\usepackage{framed}
\usepackage{graphicx}
\usepackage{xcolor}
\usepackage{listings}
\usepackage{multirow}
\usepackage{multicol}
\usepackage{indentfirst}
\usepackage{booktabs}
\usepackage{tikz}
\usepackage{tikz-cd}
\usepackage{float}
\usepackage{algorithm}
\usepackage{algorithmicx}
\usepackage{algpseudocode}
\usepackage{subfigure}
\usepackage{lineno}
\usepackage{comment}
\usepackage{setspace}
\usetikzlibrary{patterns, arrows.meta, calc}

\newcommand*{\be}[1]{\begin{equation}\label{#1}}
\newcommand*{\ee}{\end{equation}}

\DeclareMathOperator{\Span}{span}

\newtheorem{theorem}{Theorem}[section]

\newtheorem{lemma}[theorem]{Lemma}
\newtheorem{proposition}[theorem]{Proposition}

\theoremstyle{remark}
\newtheorem{remark}{Remark}[section]

\definecolor{pink}{RGB}{255,45,115}

\DeclareMathOperator{\grad}{grad}
\DeclareMathOperator{\hess}{hess}
\DeclareMathOperator{\curl}{curl}
\DeclareMathOperator{\inc}{inc}

\DeclareMathOperator{\Def}{def}
\DeclareMathOperator{\dev}{dev}
\DeclareMathOperator{\sym}{sym}
\DeclareMathOperator{\diverenge}{div}

\DeclareMathOperator{\tr}{tr}
\DeclareMathOperator{\st}{st}
\newcommand\vskw{\operatorname{vskw}}
\newcommand\mskw{\operatorname{mskw}}

\newcommand\skw{\operatorname{skw}}

\newcommand{\bM}{\mathbb M}
\newcommand{\bS}{\mathbb S}

\newcommand{\bR}{\mathbb R}
\newcommand{\bV}{\mathbb V}

\newcommand{\cH}{\mathcal H}
\newcommand{\cK}{\mathcal K}
\newcommand{\cL}{\mathcal L}
\newcommand{\cP}{\mathcal P}

\newcommand\K{\mathbb{K}}
\newcommand\M{\mathbb{M}}

\renewcommand\S{{\mathbb S}}

\newcommand\V{{\mathbb{V}}}

\newcommand\jump[1]{[\![#1]\!]}

\renewcommand{\div}{\diverenge}

\usepackage{geometry}
\newcommand{\R}{\mathbb{R}}

 \newcommand{\bs}{{\scriptscriptstyle \bullet}}
\DeclareMathOperator{\ran}{ran}
 
 \newcommand\deff{\operatorname{def}}

 \numberwithin{equation}{section}

\newcommand{\0}{\mathaccent23}

\newcommand{\lag}{\mathbf{Lag}}
\newcommand{\reg}{\mathbf{Reg}}
\newcommand{\ned}{\mathbf{Ned}}
\newcommand{\nedc}{\mathbf{Ned^c}}
\newcommand{\RM}{\mathcal{RM}}

\usepackage{cancel}
\usepackage{enumitem}

\title[Regge element, Riemann-Cartan geometry and cohomology]{Extended Regge complex for linearized Riemann-Cartan geometry and cohomology}
\author{Snorre H. Christiansen}
\address{Department of Mathematics, University of Oslo, PO Box 1053, Blindern, 0316 Oslo, Norway.}
\email{snorrec@math.uio.no}
\author{Kaibo Hu}
\address{School of Mathematics, the University of Edinburgh, James Clerk Maxwell Building, Peter Guthrie Tait Rd, Edinburgh EH9 3FD, UK.}
\email{kaibo.hu@ed.ac.uk}
\author{Ting Lin}
\address{School of Mathematics, Peking University, Beijing 100871, P.R.China}
\email{lintingsms@pku.edu.cn}

\begin{document}
\maketitle 

\begin{abstract}
We show that the cohomology of the Regge complex in three dimensions is isomorphic to $ \mathcal H^{\bs}_{dR}(\Omega)\otimes\mathcal{RM}$, the infinitesimal-rigid-body-motion-valued de~Rham cohomology. Based on an observation that the twisted de~Rham complex extends the elasticity (Riemannian deformation) complex to the linearized version of coframes, connection 1-forms, curvature and Cartan's torsion, we construct a discrete version of linearized Riemann-Cartan geometry on any triangulation and determine its cohomology. 
\end{abstract}

 \section{Introduction}

Differential complexes encode algebraic and analytic structures of many problems. Discrete versions of differential complexes play a crucial role in structure-preserving discretizations of PDEs \cite{arnold2018finite,arnold2006finite} and other applications \cite{lim2020hodge}. The elasticity (Riemannian deformation) complex is an important example, which includes the strain and stress tensors of linear elasticity, and correspondingly, metric and linearized curvature of Riemannian geometry. On the discrete level, Regge calculus  \cite{regge1961general} defines discrete metrics using the edge lengths of a triangulation. The curvature is represented by angle deficits on the codimension-2 simplices, called hinges.   Christiansen \cite{christiansen2011linearization} interpreted Regge calculus as a finite element scheme (referred to as Regge finite element below) and extended the Regge space to a discrete version of the elasticity complex (Regge complex). 
This interaction between finite elements and Regge calculus has inspired numerical schemes in solid mechanics and general relativity and studies of discretizations of curvature from a numerical analysis perspective \cite{berchenko2022finite,christiansen2013exact,gawlik2020high,gawlik2023finite-any,gawlik2023finite,gopalakrishnan2023analysis0,gopalakrishnan2023analysis,li2018regge,neunteufel2021avoiding}. 

The elasticity complex can be obtained as an example of the Bernstein-Gelfand-Gelfand (BGG) machinery \cite{arnold2021complexes,vcap2022bgg,vcap2001bernstein}. Indeed it can be derived from several copies of the de~Rham complex. An intermediate step for deriving the BGG complexes (e.g., the elasticity complex) from de~Rham complexes is the so-called twisted complexes, which include algebraic operators mixing the components. It was observed in \cite{vcap2022bgg,hu2023nonlinear} that the twisted and BGG complexes have neat correspondences in continuum mechanics:
\begin{center}
\begin{tabular}{c|c|c} 
  \hline
 & twisted complex & BGG complex\\  \hline
 1D& Timoshenko beam & Euler-Bernoulli beam\\
  2D& Reissner-Mindlin plate & Kirchhoff-Love plate\\
   3D& Cosserat elasticity & standard elasticity \\  \hline
\end{tabular}
\end{center}
Roughly speaking, the twisted complexes encode more complete models including microscopic rotational degrees of freedom. The BGG machinery is a cohomology-preserving procedure which eliminates the rotational components and gives a reduced model. The above observations concern Hodge-Laplacian problems at index 0 of complexes (Poisson-type problems, with an analogy to the de~Rham complex).  In elasticity, Kr\"oner described continuum defects as incompatibility (curvature), which is included in the elasticity complex at index 1 (Maxwell-type problems). See, for example, \cite{kroner1985incompatibility}. Roughly speaking, 0-forms correspond to compatible models, and higher-order forms correspond to defect theories of a continuum. This direction inspired us to investigate higher-order forms in the twisted complexes. In fact, formulas for defect theories of the Cosserat continuum already show connections to the twisted de~Rham complex \cite{gunther1958statik} in the context of homotopy inverses \cite{vcap2023bounded}. 

In this paper, we observe that the elasticity version of the twisted complexes in 3D encodes coframes and connection 1-forms (at index one), torsion and curvature (at index two), respectively. Therefore the entire complex complies with the observation above that the twisted complex describes more complete models: for high-order forms, {\it the (linearized) twisted complex encodes Riemann-Cartan geometry, and the BGG construction provides a cohomology-preserving machinery to eliminate the torsion to obtain Riemannian geometry}. 

At the discrete level, we extend the Regge sequence  \cite{christiansen2011linearization} to a discrete version of the twisted complex. Consequently, we obtain a possible definition of {\it discrete torsion} (see Figure \ref{fig:diagrams}). 
\begin{figure}[htbp]
{\centering
\includegraphics[width=0.9\linewidth]{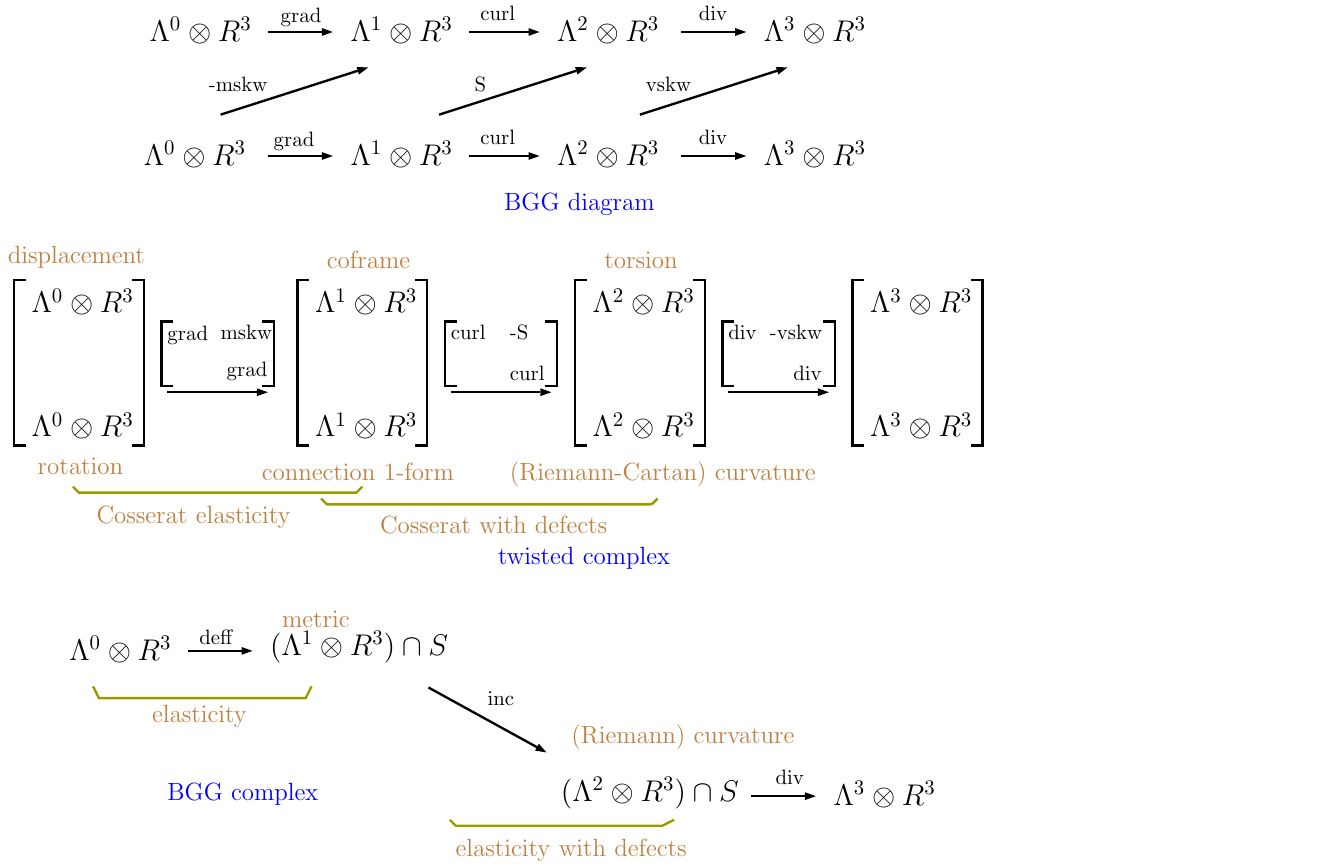} 
\caption{BGG diagram, twisted de~Rham complexes and BGG complexes. }
\label{fig:diagrams}
}
\end{figure}

We also address another open problem: the cohomology of the Regge complex. On the continuous level, it is well known that some existence and uniqueness results related to metric and curvature rely on topology. For example, the ``fundamental theorem of Riemannian geometry'' \cite{ciarlet2013linear} (reformulated as the exactness of a nonlinear elasticity complex \cite{hu2023nonlinear}) assumes simply connected domains.  Although Regge calculus/finite elements have drawn significant attention, its cohomology remains an open question, to the best of our knowledge. For Whitney forms, the degrees of freedom have a clear topological correspondence: in the lowest order case, $k$-forms are discretized on $k$-chains. For the Regge complex, the degrees of freedom also enjoy a neat discrete topological/geometric interpretation (see \cite[Section 3]{hu2023distributional}). Nevertheless, establishing its cohomology is nontrivial.  Some classical approaches, e.g., Christiansen's Finite Element Systems \cite{christiansen2010finite,christiansen2011topics} and Licht's double complex approach \cite{licht2017complexes} do not seem to provide an immediate answer. In this paper, we establish the cohomology of the Regge complex using a similar strategy developed in \cite{hu2023distributional} for the Hessian and divdiv complexes. Using the BGG machinery, we further conclude concerning the cohomology of the twisted complex (for discrete linearized Riemann-Cartan geometry). Clarifying the cohomology lays the foundation for further investigating numerical and discrete geometric problems. 

In recent years, various ``intrinsic (distributional) finite elements'' (in the sense that the elements have canonical degrees of freedom with differential form interpretations and are invariant under transforms) have been developed and implemented \cite{schoberl2014c++} for solving various problems. For example, the TDNNS method \cite{pechstein2011tangential} uses normal-normal continuous symmetric matrices to discretize the stress tensor in elasticity; the MCS method \cite{gopalakrishnan2020mass} uses tangential-normal continuous traceless matrices to discretize a stress-like variable in fluid problems. The elements developed in \cite{hu2023distributional} fit in the Hessian and divdiv complexes.  The spaces and sequences developed in this paper thus bear the potential for solving PDEs in the vein of distributional finite elements in the above examples. 

Discrete vector bundles and curvature have been investigated from finite element perspectives \cite{berchenko2022finite,berwick2021discrete,christiansen2023finite,gopalakrishnan2022analysis}. The patterns and schemes developed in this paper may also pave the way for further development of nonlinear and curved versions of discrete Riemann-Cartan geometry. Moreover, Cosserat elasticity and general relativity inspired Cartan's development of torsion \cite{hehl2007elie,scholz2019cartan}, and conversely, torsion and curvature are used to describe continuum defects (dislocation and disclination) \cite{yavari2012riemann,yavari2013riemann}. The connections to continuous and discrete complexes established in this paper may shed light on a systematic development of a cohomological approaches for modelling, analysis and computation for generalized continuum, as differential complexes and the BGG machinery may find their mirrors in mechanics categorytheoretically. However, further developments in these directions, as well as comparisons with existing literature on incorporating torsion in Regge calculus (e.g., \cite{holm2005regge,pereira2002regge,schmidt2001torsion}), are beyond the scope of this paper. 

%inspire a cohomology approach for modelling generalized , analysis and computation. 

% new developments for the analysis and computation. 

%mechanics left as future work. discrete geometry

The rest of the paper is organized as follows. In \Cref{sec:notation}, we introduce preliminaries and notation. In \Cref{sec:continuous}, we observe that the twisted complex encodes the linearized Riemann-Cartan geometry at the continuous level. In \Cref{sec:discrete}, we extend the Regge sequence to obtain a discrete version of the twisted complex. Finally, in \Cref{sec:cohomology}, the cohomology of the Regge sequence (and thus of the twisted complex) is determined. 

\section{Preliminaries and notation}\label{sec:notation}

In this section, we set up the notation that will be used throughout the paper. 
\subsection{Topology and differential operators}
Let $\Omega\subset \mathbb{R}^{3}$ be a polyhedral domain and $\Delta$ be a triangulation (simplicial complex) of $\Omega$, and 
we use $\mathsf V, \mathsf E, \mathsf F, \mathsf K$ to denote the set of vertices, edges, faces, and 3D cells, respectively. %Note that in 2D, $\mathsf F$ denotes the top (two) dimensional cells.
We use $\mathsf V_{0}$ and $\mathsf V_{\partial}$ to denote the set of internal vertices and boundary vertices, respectively (similarly for $\mathsf E$ and $\mathsf F$). {Given a simplex $\sigma\in \Delta$, let $\st(\sigma) := \cup\{K: \sigma \subset K\}$ be the local patch with respect to $\sigma$. For example, $\st(v)$ and $\st(e)$ refer to the vertex patch and the edge patch, respectively.} We use $\mathcal H^{\bs}_{dR}(\Omega)$ to denote the de Rham cohomology.

Let $K$ be a tetrahedron and $f\subset K$ be a face. We use $n_f$ to denote the unit normal vector of $f$, and $t_i^f$ are the two unit tangential vectors of $f$, which are perpendicular to each other. Let $e$ be an edge. Then $t_{e}$ is the tangent vector, and $n_{1}^{e}$ and $n_{2}^{e}$ are the two normal vectors.

{
    We introduce the relative homology with coefficient in $V$, $\mathcal H_{\bs}(\Delta,V ; \partial \Delta).$ Here, the chain group $C_k(\Delta, V; \partial \Delta)$ is generated by $k$-simplices in $\Delta\backslash \partial \Delta$. Given an interior $k$-simplex $\sigma = [ x_{0}, x_1,\cdots,  x_k]$ and an interior $(k-1)-$simplex $\tau$, define 
    \begin{equation}
        \mathcal O(\tau, \sigma) := \left\{\begin{aligned} (-1)^{j+1} & \text{ if } \tau = [ x_0,  x_1, \cdots, \widehat{ x_j}, \cdots,  x_k] \text{ for some index }j, \\ 0 & \text{ otherwise. } 
        \end{aligned}\right.
        \end{equation}
The boundary operator is then defined as
    \begin{equation}
        \partial_k \sigma := \sum_{\tau \in \Delta_{k-1}} \mathcal O(\tau, \sigma)\tau,
        \end{equation}
    and $$\mathcal H_k(\Delta, V; \partial \Delta) := \ker(\partial_{k})/ \ran(\partial_{k-1}).$$
{Clearly, $\partial_k$ induces a linear operator from $C_k(\Delta, V; \partial \Delta)$ to $C_{k-1}(\Delta, V; \partial \Delta)$.}
We have the following results on the relative homology.
\begin{theorem}
\begin{enumerate}
    \item (Universal Coefficient Theorem). $\mathcal H_k(\Delta, V; \partial \Delta) \cong \mathcal H_k(\Delta; \partial \Delta) \otimes V.$
    \item (de Rham Theorem). In three dimensions, $\mathcal H_k(\Delta; \partial \Delta) \cong \mathcal H_{dR}^{3-k}(\Omega)$.
\end{enumerate}
\end{theorem}
}

Throughout the paper, we use the convention that differential operators $\grad$, $\curl$ and $\div$ are row-wise when they are applied to a matrix field. We also use the nabla notation: for example,  column-wise $(\nabla v)_{ij}=\partial_i v_j$ and row-wise $(v\nabla)_{ij}=\partial_jv_i$. Therefore $\grad v=v\nabla$.  The definition for column-wise $\nabla\times w$, $\nabla\cdot w$ and row-wise $w\times \nabla$, $w\cdot \nabla$ is similar. 

Following \cite{arnold2021complexes,vcap2022bgg}, we use vector/matrix proxies: \;t{Specify the 3D.}
\begin{table}[h!]
\begin{center}
\begin{tabular}{c|c}
$\mathbb V$ & $\mathbb R^3$  \\
$\mathbb M$ &the space of all $3\times 3$-matrices\\
$\mathbb S$ & symmetric matrices\\
$\mathbb K$ & skew symmetric matrices\\
$\mathbb T$ & trace-free matrices\\
$\skw: \M\to \K$ & skew symmetric part of a matrix\\
$\sym: \M\to \S$ & symmetric part of a matrix\\
$\tr:\M\to\R$ & matrix trace\\
$\iota: \R\to \M$  & the map $\iota u:= uI$ identifying a scalar with a scalar matrix\\
$\dev:\mathbb{M}\to \mathbb{T}$ & deviator (trace-free part of a matrix) given by $\dev u:=u-1/3 \tr (u)I$\\
\end{tabular}
\end{center}
\end{table}
  %  Put $\mathbb V:=\mathbb R^n$, let $\mathbb M$ be the space of all $n\times n$-matrices and, $\mathbb S$ and $\mathbb K$ and $\mathbb T$ be the subspaces of matrices that are symmetric, skew-symmetric and trace-free, respectively. Now we consider the following basic linear algebraic operations: $\skw: \M\to \K$ and $\sym:\M\to\S$ are the skew and symmetric part operators; $\tr:\M\to\R$ is the matrix trace; $\iota: \R\to \M$ is the map $\iota u:= uI$ identifying a scalar with a scalar matrix;  $\dev:\mathbb{M}\to \mathbb{T}$ given by $\dev u:=u-1/n \tr (u)I$ is the deviator, or trace-free part.

In three space dimensions, we identify a skew-symmetric matrix with a vector,
$$
 \mskw\left ( 
\begin{array}{c}
v_{1}\\ v_{2}\\ v_{3}
\end{array}
\right ):= \left ( 
\begin{array}{ccc}
0 & -v_{3} & v_{2} \\
v_{3} & 0 & -v_{1}\\
-v_{2} & v_{1} & 0
\end{array}
\right ).
$$
Consequently, we have   $\mskw(v)w = v\times w$ for $v,w\in\V$.  We also define
$\vskw=\mskw^{-1}\circ \skw: \M\to \V$. We further define the Hessian operator $\hess:=\grad\circ\grad$, {the linearized deformation operator $\deff = \sym \circ \grad$,} and for any matrix-valued function $u$, $\mathcal{S}u:=u^{T}-\tr(u)I$.  Define the incompatibility operator (linearized Einstein tensor) $\inc :=\curl \circ \mathrm{T}\circ\curl $, i.e., $\inc \sigma:=\nabla\times \sigma\times\nabla$. %In our convention, the proxies of exterior derivatives ($\grad$, $\curl$, $\div$, etc.) act column-wise, and so do the Poincar\'e operators for the de~Rham complexes. 

We introduce Dirac measures on (sub)simplices. Given a simplex $\sigma$, we define the scalar, vector and matrix-valued delta $\delta_{\sigma}$,  $\delta_{\sigma}[a]$ and $\delta_{\sigma}[ A]$ by
$$
\langle \delta_{\sigma}, u \rangle = \int_{\sigma} u,\quad \langle \delta_{\sigma}[a],  v \rangle  = \int_{\sigma}  a \cdot  v, \quad \langle \delta_{\sigma}[ A],  W \rangle = \int_{\sigma}  A :  W,
$$
for smooth scalar, vector and matrix-valued functions $u$, $ v$, and $W$, respectively. Here $ a\in \mathbb V$ and $A\in \mathbb M$.

\subsection{Finite element spaces and the Regge complex}
 \begin{figure}[htbp]%\centering
\includegraphics[width=0.8\linewidth]{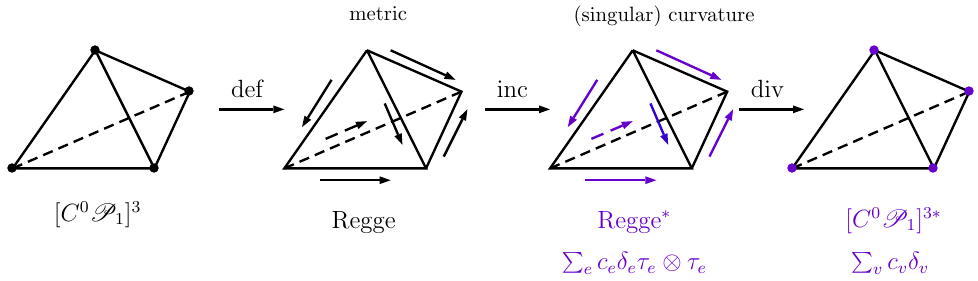} 
\caption{Regge complex \eqref{cplx:regge} (\cite{christiansen2011linearization}). }
\label{fig:regge}
\end{figure}
We recall the Regge complex (see \Cref{fig:regge}):
\begin{equation}
\label{cplx:regge}
\begin{tikzcd}
0 \ar[r] & \lag \ar[r, "\Def"] & \reg \ar[r,"\inc"] & (\reg_0)' \ar[r,"\div"] & (\lag_0)' \ar[r] & 0.
\end{tikzcd}
\end{equation}

%and the homogeneous version. 
Here, the space $\lag$ stands for the lowest-order Lagrange element space, defined as 
\begin{equation}
\lag := \{ u \in C^0(\Omega; \bV) : u|_K \in \cP_1(K)\otimes  \bV \text{ for each cell }K\}.
\end{equation}
%\kh{we need to be consistent: either use tensor notation $P_1\otimes \mathbb{V}$, or $P_1(\Omega; \mathbb{V})$.}\lt{I prefer the latter one.} \kh{But now we have consistently used the former?}\lt{For the function space, I use the latter one. For notation like $\mathcal P_1 \otimes \mathbb V$, I prefer the tensor notation.}
The degrees of freedom of $\lag$ are $u \mapsto u(x)$ at each vertex $x \in \mathsf V$. 
The space $\reg$ represents the Regge element space, consisting of piecewise constant symmetric-matrix-valued functions with  tangential-tangential interelement continuity, defined as 
\begin{equation}
\begin{split}
\reg := \{ \sigma \in L^2(\Omega; \bS) :& \sigma|_K \in \mathbb S \text{ for each cell } K, \\ & n_{f} \times \sigma \times n_{f} \text{ is continuous on each internal face } f \}.
\end{split}
\end{equation}
The degrees of freedom of $\reg$ are $\sigma \mapsto \int_e (t_e \cdot \sigma \cdot t_e)|_e$ for each $e \in \mathsf E$; see \cite{christiansen2011linearization}. 
%\lt{Consistency:$f$ or $F$?}\kh{let's change to $f$ then - distinguishable from $\mathsf{F}$.}
The space $\reg$ has weak regularity (with continuous tangential-tangential components across the boundaries of cells) and the remaining two spaces consist of distributions. Define
$$(\lag_0)' := \Span\{ \delta_{x}[v]: v \in \bV, x \in \mathsf V_0\},$$
where $\delta_{x}[v]: u \mapsto (u \cdot v)(x)$ is the vertex Dirac delta for $v \in \mathbb V$ and vertex $x$; and
$$(\reg_0)' := \Span\{ \delta_{e}[t_e t_e^T]: e \in \mathsf E_0\},$$
where $\delta_{e}[t_e t_e^T]: \sigma \mapsto \int_{e} t_e \cdot \sigma \cdot t_e$ for an edge $e$ is the tangential-tangential edge delta.

The operators in  \eqref{cplx:regge} are defined in the sense of distributions, i.e.,  $\Def$, $\inc$ and $\div$ are given by
%Next, we specify the definition of the differential operators in the distributional sense. The definitions are standard, namely,
$$
\langle \Def u, \sigma \rangle = -\langle u, \div \sigma \rangle, \forall \sigma \in C^{\infty}_c(\Omega; \mathbb{S}),
$$
$$\langle \inc \sigma, \tau \rangle = \langle \sigma, \inc \tau \rangle, \forall \tau \in C_c^{\infty}(\Omega; \mathbb S),$$
$$\langle \div \sigma, u \rangle = - \langle \sigma, \deff u \rangle, \forall u \in C_c^{\infty}(\Omega; \mathbb V).$$
 In \cite{christiansen2011linearization}, it was shown that the sequence \eqref{cplx:regge} is a complex. % In this paper, we shall show that the cohomology is isomorphic to $\mathcal{RM} \otimes \cH^{\bs}_{dR}(\Omega)$. Here $\RM := \ker(\Def) = \{\bm a + \bm b \times \bm x: \bm a, \bm b \in \mathbb V\}$ is the rigid motion in three dimensions.

For later use, we define some vector and finite element spaces. Let $\RM := \ker(\Def) = \{x \mapsto a +  b \times  x:  a, b \in \mathbb V\}$ be the space of infinitesimal rigid motions in three dimensions. The first-type N\'ed\'elec space of the lowest order \cite{nedelec1980mixed} is 
\begin{equation}
\label{eq:ned}
    \ned := \{u : u|_{K} \in \RM \text{ for each cell } K : u \times n_f \text{ is continuous on each face }f\}.
    \end{equation}
The degrees of freedom are $u\mapsto \int_e u\cdot t_e$. 
    The second-type N\'ed\'elec space of the lowest order \cite{nedelec1980mixed} is 
\begin{equation}
\label{eq:nedc}
\nedc : = \{u: u|_K \in \cP_1 \otimes \mathbb V \text{ for each cell } K,  u \times n_f \text{ is continuous on each face }f\}.
\end{equation}
The degrees of freedom can be given by 
$u \mapsto \int_{e} (u \cdot t_e)p,$ for all $p \in \mathcal P_1(e).$

%\kh{define spaces with subscripts 0.}

\section{Linearized Riemann-Cartan geometry and twisted complexes}\label{sec:continuous}

In this section, we observe that the elasticity version of the twisted de~Rham complex corresponds to the linearized Riemann-Cartan geometry \cite{sharpe2000differential}. We mainly follow the notation of \cite{hehl2007elie}.

Recall that the BGG diagram for deriving the elasticity complex is 
\begin{equation}\label{bgg-continuous}
\begin{tikzcd}
0 \arrow{r} &\Lambda^{0, 1} \arrow{r}{\grad} &\Lambda^{1, 1} \arrow{r}{\curl} &\Lambda^{2, 1} \arrow{r}{\div} &\Lambda^{3, 1} \arrow{r}{} & 0\\
0 \arrow{r} &\Lambda^{0, 2} \arrow{r}{\grad} \arrow[ur, "-\mskw"]&\Lambda^{1, 2} \arrow{r}{\curl} \arrow[ur, "\mathcal{S}"]&\Lambda^{2, 2}\arrow{r}{\div}\arrow[ur, "2\vskw"] &\Lambda^{3, 2}  \arrow{r}{} & 0.
 \end{tikzcd}
\end{equation}
Here $\Lambda^{i, j}$ is the space of $j$-form-valued $i$-forms. When vector/matrix proxies are used, $\Lambda^{0, j} $ and $\Lambda^{3, j}$ corresponds to vector-valued functions, while $\Lambda^{1, j} $ and $\Lambda^{2, j}$ are matrix-valued functions, for $j = 1,2$. Then the exterior derivatives and the connecting algebraic maps have the form shown in the diagram \eqref{bgg-continuous}.

The twisted complex derived from the diagram \eqref{bgg-continuous} is 
\begin{equation} \label{twisted-continuous}
\begin{tikzcd}
0 \arrow{r} &
\left (
\begin{array}{c}
\Lambda^{0, 1}\\
\Lambda^{0, 2}
\end{array}
\right )
 \arrow{r}{\mathcal{A}^0} &\left (
\begin{array}{c}
\Lambda^{1, 1}\\
\Lambda^{1, 2}
\end{array}
\right ) \arrow{r}{\mathcal{A}^1} &\left (
\begin{array}{c}
\Lambda^{2, 1}\\
\Lambda^{2, 2}
\end{array}
\right )\arrow{r}{\mathcal{A}^2} &\left (
\begin{array}{c}
\Lambda^{3, 1}\\
\Lambda^{3, 2}
\end{array}
\right )  \arrow{r}{} & 0,
 \end{tikzcd}
\end{equation}
where $\mathcal{A}^\bs=\left (
\begin{array}{cc}
d^\bs & -S^\bs\\
0 & d^\bs 
\end{array}
\right )$ are the twisted operators, i.e., 
$$
\mathcal{A}^0=\left (
\begin{array}{cc}
\grad & \mskw\\
0 & \grad 
\end{array}
\right ), \quad 
\mathcal{A}^1=\left (
\begin{array}{cc}
\curl & -\mathcal{S}\\
0 & \curl 
\end{array}
\right ), \quad \mathcal{A}^2=\left (
\begin{array}{cc}
\div & -2\vskw\\
0 & \div 
\end{array}
\right ).  
$$
The fact that \eqref{twisted-continuous} is a complex is a consequence of the (anti)commutativity of the diagram \eqref{bgg-continuous} (formally, $dS=-Sd$, with $d^\bs$ and $S^\bs$ of corresponding degrees).

In the rest of this section, we will fit various quantities and operators of linearized Riemann-Cartan geometry in the twisted complex \eqref{twisted-continuous}.

Let $e_{\alpha}, \alpha=1, \cdots, n$ be a frame, and $\theta^{\beta}, \beta=1, \cdots, n$ be its dual basis, i.e., a {\it coframe}. The connection is introduced as 1-forms: $\Gamma_{\alpha}^{\beta}:=\Gamma_{i\alpha}^{\beta}dx^{i}$. Here, the Latin alphabet indicates coordinate indices, while Greek letters indicate (anholonomic) frame indices. Therefore $\Gamma_{i\alpha}^{\beta}dx^{i}$ corresponds to a transform from a frame basis to a frame basis. Given a coframe and a connection (both 1-forms), one can define torsion 2-form \cite[(3)]{hehl2007elie} and curvature 2-form \cite[(4)]{hehl2007elie}: 
\begin{equation}\label{def:T}
T^{\alpha}:=D\theta^{\alpha}:=d\theta^{\alpha}+\Gamma_{\beta}^{\alpha}\wedge \theta^{\beta}, 
\end{equation}
\begin{equation}\label{def:R}
R_{\alpha}^{\beta}:=d\Gamma_{\alpha}^{\beta}+\Gamma_{\gamma}^{\beta}\wedge \Gamma_{\alpha}^{\gamma}.
\end{equation}
Here $d$ is the exterior derivative and $D$ is the covariant exterior derivative. 
 At this point, we should understand $\Gamma_{\alpha}^{\beta}$ as a tensor with three indices, as  $\Gamma_{\alpha}^{\beta}=\Gamma_{i\alpha}^{\beta}dx^{i}$. {Assuming that we have a Riemann-Cartan connection}, i.e., we have a metric $g$ and the connection is compatible with $g$ ($Dg=0$). By the definition of $D$, we have, for the symmetric part,  $\Gamma_{(\alpha\beta)}=\frac{1}{2}dg_{\alpha\beta}$, where $\Gamma_{\alpha\beta}=g_{\alpha\gamma}\Gamma_{\beta}^{\gamma}$ \cite[(8)]{hehl2007elie}. Moreover, we {assume that the (co)frame is orthonormal}, i.e., $g_{\alpha\beta}$ corresponds to the identity matrix (or with some $-1$ in the Minkowski case). Therefore, $\Gamma_{(\alpha\beta)}=\frac{1}{2}dg_{\alpha\beta}=0$. This implies that $\Gamma_{\alpha\beta}$ is skew-symmetric with its two indices. Now, {to see correspondence to \eqref{bgg-continuous} where the quantities are matrix-valued functions,} we can regard $\Gamma_{i\alpha\beta}:=g_{\alpha\gamma}\Gamma^{\gamma}_{i\beta}$ as a tensor with two indices (matrices) as we can compress $\alpha, \beta$ to one index.% and leave another form index as it is. 

Now we consider the linearization of \eqref{def:T} and \eqref{def:R} around trivial connections ($\Gamma_{\beta}^{\alpha}=0$) and $\theta^{\alpha}=dx^{\alpha}$. Under linearization, we can freely lower or raise indices. This means that we can now regard $\Gamma_{\alpha}^{\beta}$ as a matrix.  The first term in \eqref{def:R} is just $d$ of a skew-symmetric matrix-valued 1-form, which corresponds to a curl. The second term $\Gamma_{\gamma}^{\beta}\wedge \Gamma_{\alpha}^{\gamma}$ is nonlinear and vanishes in the linearization. Therefore, we write $\eqref{def:R}$ as $R_{\alpha\beta}=\curl \Gamma_{\alpha\beta}$. Similarly, the first term in \eqref{def:T}, $d\theta^{\alpha}$ corresponds to a curl of a matrix. The second term $\Gamma_{\beta}^{\alpha}\wedge \theta^{\beta}$ is nonlinear and two terms come out from the linearization: 
\begin{equation}\label{linearization}
\0\Gamma_{\beta}^{\alpha}\wedge \theta^{\beta}+\Gamma_{\beta}^{\alpha}\wedge \0\theta^{\beta}, 
\end{equation}
where $\0\Gamma_{\beta}^{\alpha}=0$ and $\0\theta^{\beta}=dx^{\beta}$ indicate the background quantities. The first term  vanishes as $\0\Gamma_{\beta}^{\alpha}=0$, and therefore only $\Gamma_{\beta}^{\alpha}\wedge \0\theta^{\beta}$ remains.  {Below we show that the linearization of the term $\Gamma_{\beta}^{\alpha}\wedge \0\theta^{\beta}$ exactly corresponds to $\mathcal{S}\Gamma_{\beta}^{\alpha}$ with connections to the diagram \eqref{bgg-continuous}. }
In fact, given $\Gamma_{i\beta}^{\alpha}dx^{i}$ which corresponds to a skew-symmetric matrix (suppressing $\alpha$ and $\beta$), we want to know what is the operation that takes it to $\Gamma_{\beta}^{\alpha}\wedge \0\theta^{\beta}$. %\kh{This map will corresponds to the vector/matrix proxy of the connecting maps in the twisted complex \ref{}.} 
Since
$$
\Gamma_{\beta}^{\alpha}\wedge dx^{\beta}=\Gamma_{i\beta}^{\alpha}dx^{i}\wedge dx^{\beta},
$$
 the operation corresponds to skew-symmetrization of $\Gamma_{i\beta}^{\alpha}$ with respect to $i$ and $\beta$. %That is, to get the matrix proxy of the 2-form $\Gamma_{\beta}^{\alpha}\wedge dx^{\beta}$, we first expand the matrix proxy of the 1-form $\Gamma_{i\beta}dx^{i}$ by an $\mskw$ (to get three indices $\Gamma_{i\beta}^{\alpha}$), then take skew-symmetrization on $i$ and $\beta$. These steps are summarized as follows:
{With vector/matrix proxies, the above map is summarized as follows, with $\Gamma_{i\alpha\beta}$ and the skew-symmetrization $\epsilon^{i\beta\gamma}\Gamma_{i\alpha\beta}$ viewed as matrices:
$$
T_{i\mu}:=\frac{1}{2}\epsilon_{\mu}^{~~\alpha\beta}\Gamma_{i\alpha\beta}\mapsto \Gamma_{i\alpha\beta}=\epsilon^{\mu\alpha}_{\quad\beta}T_{i\mu}\to  \epsilon^{i\beta\gamma}\Gamma_{i\alpha\beta}=T^{\alpha\gamma}-(\tr T)\delta^{\alpha\gamma}=(\mathcal{S}T)^{\alpha\gamma}.% \varepsilon^{\alpha\beta\mu}T_{i\mu}\mapsto \epsilon_{l\beta}^{\quad i}\varepsilon^{\alpha\beta\mu}T_{i\mu},
$$
Here we start with a matrix $T_{i\mu}$, which is the axial vector version of $\Gamma_{i\alpha\beta}$ ($\vskw$ with respect to $\alpha, \beta$). The first map recovers $\Gamma_{i\alpha\beta}$ from $T_{i\mu}$. The second map is to take the skew-symmetrization with respect to $i, \beta$ and identify the result with an axial vector ($\vskw$). The resulting map from $T_{i\mu}$ to $ \epsilon^{i\beta\gamma}\Gamma_{i\alpha\beta}$ exactly corresponds to the $\mathcal{S}$ operator in \eqref{bgg-continuous}. 
}

%where the first arrow expands one index $i$ to two indices $\alpha\beta$ ($\mskw$), and the second arrow skew-symmetrizes $i$ and $\beta$. This leads to 
%$$
%\epsilon_{l\beta}^{\quad i}\varepsilon^{\alpha\beta\mu}T_{i\mu}=(\delta_{l}^{\mu}\delta^{i\alpha}-\delta_{l}^{\alpha}\delta^{i\mu})T_{i\mu}=T^{\alpha}_{ \quad l}-(\tr T)\delta_{l}^{\alpha},
%$$
%which corresponds to $T^{T} -\tr(T)I$. This is the $\mathcal{S}$ operator in the BGG diagram for the elasticity complex. 

To summarize, assuming that we are given  a background consisting of an orthonormal coframe $\0\theta^{\alpha}$ and  a Riemann-Cartan connection $\0\Gamma^{\alpha}_{\beta}$, 
which are both 1-forms, then the linearized torsion 2-form and the curvature 2-form operators are given by (in terms of proxies)
\begin{align}
T&=\curl {\theta}+\mathcal{S}\Gamma, \\
R&=\curl \Gamma,
\end{align}
{or in a more compact form, 
\begin{equation}\label{theta-Gamma-eqn}
\left ( 
\begin{array}{c}
T\\
R
\end{array}
\right )=\left ( 
\begin{array}{cc}
d & \mathcal{S}\\
0 & d
\end{array}
\right )\left ( 
\begin{array}{c}
\theta\\
\Gamma
\end{array}
\right ).
\end{equation}
The right hand side is the operator $\mathcal{A}^{1}$ in \eqref{twisted-continuous} with an adjusted sign.

%Note that $\theta$ and $\Gamma$ here perturbations in the linearization \eqref{linearization}, which are different from the 

With \eqref{theta-Gamma-eqn}, we observe the correspondence between linearized Riemann-Cartan geometry and the twisted complex \eqref{twisted-continuous}: $\Lambda^{1, 1}$ corresponds to coframe;  $\Lambda^{1, 2}$ corresponds to connection 1-forms; $\Lambda^{2, 1}$ corresponds to torsion; $\Lambda^{2, 2}$ corresponds to curvature. 
}
%The right hand side is given by the operator % corresponds to the operator

% in the twisted complex with an adjustment of the signs. 
%  \kh{to double check the sign. If the sign was wrong, we can adjust the signs. }\lt{I}
  %(up to a sign in front of $S$, but either because I had a wrong sign somewhere, or the BGG diagrams can be easily adjusted to flip the signs).

\section{Extended Regge complexes and discrete Riemann-Cartan geometry}\label{sec:discrete}

{In Section \ref{sec:continuous}, we observed that the twisted complex \eqref{twisted-continuous} corresponds to quantities and operators in the linearized Riemann-Cartan geometry. The twisted complex encodes additional information compared to the BGG complex. In the current context, torsion appears in the twisted complex, while in the BGG construction, the torsion part is eliminated. Thus the resulting BGG complex (the elasticity complex) only describes Riemannian geometry.  

Regge finite elements and the corresponding sequence are a neat discrete version of the elasticity complex (Riemannian geometry). A natural question arises: is there a discrete version of the twisted complex \eqref{twisted-continuous} which extends the Regge sequence in the sense that the Regge sequence can be derived from it on the continuous level? If such an extension exists, the result would be a natural candidate for a discrete version of Riemann-Cartan geometry, in particular, the torsion tensors. 

In this section, we construct such an extension.
}
%In this section, we extend the Regge complex to the full twisted complex.
 The spaces are summarized in the following diagram:
%We construct the BGG diagram for the Regge finite element complex as the following:
\begin{equation}\label{bgg-regge}
\begin{tikzcd}
0 \arrow{r} &V_{h}^{0} \arrow{r}{\grad} &V_{h}^{1} \arrow{r}{\curl} &V_{h}^{2} \arrow{r}{\div} &V_{h}^{3}  \arrow{r}{} & 0\\
0 \arrow{r} &W_{h}^{0} \arrow{r}{\grad} \arrow[ur, "-\mskw"]&W_{h}^{1}  \arrow{r}{\curl} \arrow[ur, "\mathcal{S}"]&W_{h}^{2}  \arrow{r}{\div}\arrow[ur, "2\vskw"] &W_{h}^{3}  \arrow{r}{} & 0.
 \end{tikzcd}
\end{equation}
The twisted complex derived from the diagram \eqref{bgg-regge} is thus 
\begin{equation}\label{twisted-regge}
\begin{tikzcd}
0 \arrow{r} &
\left (
\begin{array}{c}
V_{h}^{0}\\
W_{h}^{0} 
\end{array}
\right )
 \arrow{r}{\mathcal{A}^0} &\left (
\begin{array}{c}
V_{h}^{1}\\
W_{h}^{1} 
\end{array}
\right ) \arrow{r}{\mathcal{A}^1} &\left (
\begin{array}{c}
V_{h}^{2}\\
W_{h}^{2} 
\end{array}
\right )\arrow{r}{\mathcal{A}^2} &\left (
\begin{array}{c}
V_{h}^{3}\\
W_{h}^{3} 
\end{array}
\right )  \arrow{r}{} & 0.
 \end{tikzcd}
\end{equation}

The spaces in the $(V^{\bs}, d^{\bs})$ complex are as follows:
\begin{itemize}[leftmargin=*]
\item[-] $V_{h}^{0}$: the vector-valued linear Lagrange space $\lag$. The degrees of freedom are given by the vertex evaluation of each component.
\item[-] $V_{h}^{1}$: the Regge element enriched with piecewise constant skew-symmetric matrices, which can be written as 
{\begin{equation}
    \reg \oplus \{ \mskw(c): c \in L^2(\Delta; \mathbb V) \text{ is piecewise constant}\}.
\end{equation}
 } Another way to obtain this space is to use piecewise constants (including both symmetric and skew-symmetric modes) as the shape functions and use the edge evaluation $u\mapsto \int_{e}t_{e}\cdot u\cdot t_{e}$ and three interior modes (against constant skew-symmetric matrices) as the degrees of freedom. Note that the symmetric part still has tangential-tangential continuity across faces, while the skew-symmetric part does not. The dimension of $V_h^1$ is $\#_{\mathsf E} + 3\#_{\mathsf K}$.
\item[-] $V_{h}^{2}:=\mathrm{span}_{f\in\mathsf{F}_{0}}\{\delta_{f}[nt_{1}^T], \delta_{f}[nt_{2}^T], \delta_{f}[t_{1}t_{1}^T+t_{2}t_{2}^T]\}$. Here % \kh{use $\Delta$ or $\mathsf E$ etc.? all faces or interior faces?}\lt{Using $E$ and for all interior faces. }
$$
\langle \delta_{f}[nt_{i}^T], \phi\rangle :=\int_{f}n_{f}\cdot \phi\cdot t_{i}^{f}\, dx,
$$
$$
\langle \delta_{f}[t_{1}t_{1}^T+t_{2}t_{2}^T], \phi\rangle :=\int_{f}t_{1}^{f}\cdot \phi\cdot t_{1}^{f}+t_{2}^{f}\cdot \phi\cdot t_{2}^{f}\, dx.
$$
\item[-] $V_{h}^{3}$ consists of edge normal deltas, i.e., $V_{h}^{3}:=\mathrm{span}_{e\in \mathsf{E}_{0}}\{\delta_e[{n_{i}}]: ~i=1, 2\}$.
\end{itemize}
The lower complex $(W_{h}^{\bs}, d^{\bs})$ consists of three copies of the lowest order distributional de~Rham elements \cite{braess2008equilibrated,licht2017complexes} (see \Cref{fig:distributional-deRham}). Specifically,
 $W_h^0$ is the piecewise constant space, and %$W_h^1$ consists of normal deltas on the internal faces, $W_h^2$ consists of tangential deltas on the internal edges, and finally, $W_h^3$ consists of deltas at the vertices, respectively, i.e., 
$$
W_h^1:=\mathrm{span}_{f\in \mathsf{F}_{0}}\{\delta_{f}[c \otimes n_{f}]: c \in \mathbb V \},\quad W_h^2:=\mathrm{span}_{e\in \mathsf{E}_{0}}\{  \delta_{e}[c \otimes t_{e}]: c \in \mathbb V\},
$$
$$W_h^3:=\mathrm{span}_{x\in \mathsf{V}_{0}}\{\delta_{x}[c]: c \in \mathbb V\}.$$
By the definition of distributional derivatives, the sequence $(W_{h}^{\bs}, d^{\bs})$ is a complex. This can also be verified using the fact that $(W_{h}^{\bs}, d^{\bs})$ is the dual of the finite de~Rham complex consisting of the lowest order Whitney forms \cite{arnold2006finite}.
% \lt{I suggest writing down the explicit form} \kh{done. I also cited Braess and Schoeberl. But do we need more explanations on the notations above? }
%   \lt{I think no need for the notation.  But do we need to exaplain why it is a complex?}\kh{brief explanations added. }%More specifically, each copy consists of piecewise constant, face normal deltas, edge tangential deltas and vertex deltas, respectively.  
 \begin{figure}[htbp]%\centering
\includegraphics[width=0.8\linewidth]{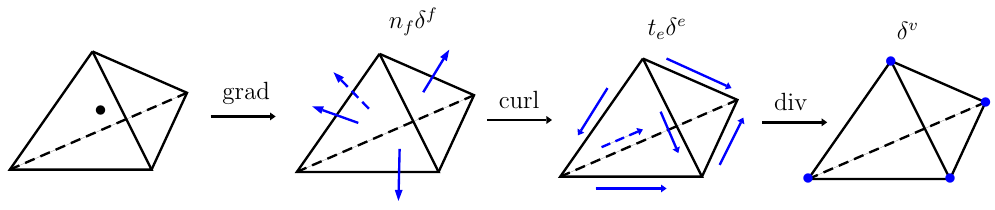} 
\caption{Distributional de~Rham complex \cite{braess2008equilibrated,licht2017complexes}, lower row $(W_h^\bs, d^\bs)$ of \eqref{bgg-regge}.}
\label{fig:distributional-deRham}
\end{figure}

Below we check that the differential and algebraic operators are well defined, i.e., they map to the correct spaces. %}\kh{to write the paragraphs below in better formats.} %In the computation below, we use the convention that differential operators act row-wise (from the right). 

\begin{enumerate}[leftmargin=*]
\item {\bf \noindent $\grad V_{h}^{0}\subset V_{h}^{1}$.} 

The edge tangential-tangential continuity comes from the fact that $V_h^{0}$ is continuous.

\item {\bf \noindent $\curl V_{h}^{1}\subset V_{h}^{2}$.}

If $u$ is a piecewise constant, then by definition, for $\varphi \in C^{\infty}_c(\Omega; \bM)$, 
\begin{align*}
\langle \curl u, \varphi \rangle = \langle u, \curl\varphi\rangle& =\sum_{K}[(\curl u, \varphi)+(u, \varphi\times n_f)_{\partial K}]
\\ & = -\sum_{K}(u\times n_f, \varphi)_{\partial K}= - \langle\sum_{f\in \mathsf F_0}\delta_{f}\Big[\jump{u \times n_f}_f\Big] , \varphi\rangle.
\end{align*}
{Here $\jump{c}_f = \sum_{K \supset f} \mathcal O(f, K) c|_{K}$ is the jump across the face $f$, which is frequently used in the DG context.} 
For $\sigma \in \reg$, the tangential-tangential components are continuous, then the above calculation shows that the curl of $\reg$ is spanned by $\delta_f[n_f \otimes t_i^f]$.

Next we compute the remaining term, namely, $\curl \mskw c$, where $c$ is a piecewise constant vector. %\kh{the previous notation $\curl \mskw(\mathbb{R}^{3})$ seems not proper for piecewise constant.  do we have a reasonable notation for p.w. constants?}\lt{I change some notation, cf. (4.3).}
 Then using the above formula,  we have
$$
\langle \mskw (c), \curl \phi \rangle=\sum_{\partial{K}}-\langle \mskw(c)\times n_f, \varphi\rangle.
$$
Moreover,  we have $\mskw(c)\times n_f=-(c\cdot n_f)I+n_f\otimes c$.
%\begin{align*}
%(\mskw(c)\times n)_{is}=\epsilon_{s}^{~~jl}\epsilon_{ijk}c^{k}n_{l}=-(\delta_{is}\delta_{k}^{l}-\delta_{ks}\delta_{i}^{l})c^{k}n_{l}=-c^{l}n_{l}\delta_{is}+c_{s}n_{i}.
%\end{align*}
%This implies that $\mskw(c)\times n=-(c\cdot n)I+n\otimes c$.
Therefore, 
\begin{align*}
-[c\cdot n_f]I+n_{f}\otimes [c]&=-[c\cdot n_f](n_f\otimes n_f+t_{1}^f\otimes t_{1}^f+t_{2}^f\otimes t_{2}^f)+[c\cdot n_f]n_f\otimes n_f\\&\quad\quad +[c\cdot t_{1}^f]n_f\otimes t_{1}^f+[c\cdot t_{2}^f]n_f\otimes t_{2}^f
\\ &=-[c\cdot n_f](t_{1}^f\otimes t_{1}^f+t_{2}^f\otimes t_{2}^f)+[c\cdot t_{1}^f]n_{f}\otimes t_{1}^f+[c\cdot t_{2}^f]n_{f}\otimes t_{2}^f.
\end{align*}
This implies that {$\curl \mskw c$} is in the  span of $\delta_F[{n \otimes t_{i}^f}]$ and $\delta_F[{t_{1}^f \otimes t_{1}^f+t_{2}^f \otimes t_{2}^f}]$ on faces.

\item {\bf \noindent $\div V_{h}^{2}\subset V_{h}^{3}$.}

For $c \in \mathbb V$, $t = t_1^f$ or $t_2^f$,
$$
\langle \div \delta_{f}[c t^T], \varphi \rangle:=\langle  \delta_{f}[ct^T], \grad\varphi \rangle=\int_{f}c\cdot (\varphi\nabla)\cdot t. 
$$
We can integrate by parts the tangential derivatives to edges and this becomes normal deltas on edges. 

\item 
{\bf \noindent $\mskw W_{h}^{0}\subset V_{h}^{1}$.} Obvious.

\item
{\bf \noindent $\mathcal{S} W_{h}^{1}\subset V_{h}^{2}$.} 
\begin{align*}
\langle &\mathcal S \delta_{f}[c\otimes n_f], \varphi \rangle:=\langle  \delta_{f}[c\otimes n_f], \varphi^{T}-\tr(\varphi)I \rangle\\&
=\int_{f}(n_f\cdot \varphi \cdot t_{1}^f)(c\cdot t_{1}^f) +(n_f\cdot \varphi \cdot t_{2}^f)(c\cdot t_{2}^f) +(n_f\cdot \varphi \cdot n_f)(c\cdot n_f) 
\\ &\qquad -(c\cdot n_f)(n_f\cdot \varphi \cdot n_f+t_{1}^f\cdot \varphi \cdot t_{1}^f+t_{2}^f\cdot \varphi \cdot t_{2}^f)
\\&
=\int_{f}(n_f\cdot \varphi \cdot t_{1}^f)(c\cdot t_{1}^f) +(n_f\cdot \varphi \cdot t_{2}^f)(c\cdot t_{2}^f)  -(c\cdot n_f)(t_{1}^f\cdot \varphi \cdot t_{1}^f+t_{2}^f\cdot \varphi \cdot t_{2}^f).
\end{align*}

\item 
{\bf \noindent $\vskw W_{h}^{2}\subset V_{h}^{3}$.}
$$
\langle \vskw \delta_e[c\otimes t_e] , \varphi \rangle:=\langle  \delta_e[c\otimes t_e], \mskw\varphi \rangle=\int_{e}c\cdot \mskw \varphi\cdot t_e=\int_{e}c\cdot \varphi\times t_e. 
$$
This corresponds to the normal component of $\varphi$ on edges.
\end{enumerate}

As we now have a BGG diagram \eqref{bgg-regge}, roughly speaking, we can eliminate the skew-symmetric components from the diagram and obtain the Regge sequence from \eqref{bgg-regge}. Nevertheless, to claim that the cohomology of \eqref{bgg-regge} is isomorphic to the cohomology of the Regge sequence (which we will derive in the next section), we should make a remark on the technicality. In \cite{arnold2021complexes,vcap2022bgg}, the BGG complexes involve spaces $\ran(S^{\bs})^{\perp}$, the orthogonal complement of the range of $S^{\bs}$, and operators $P_{\ran(S^{\bs})^{\perp}}$, the orthogonal projection to the corresponding spaces. We have not defined any inner product or orthogonality in the setting of \eqref{bgg-regge} yet. However, such orthogonality in \cite{arnold2021complexes,vcap2022bgg} is only to simplify the presentation, which is not essential in the spirit of \cite{vcap2001bernstein}. This can be seen from the presentation in \cite{vcap2022bgg}, where the pseudo-inverse of $S^{\bs}$, namely, the $T^{\bs}$ operators, are introduced. The identity $\ran(S^{\bs})^{\perp}=\ker(T^{\bs})$ implies that we can get rid of the inner product by replacing $\ran(S^{\bs})^{\perp}$ by $\ker(T^{\bs})$ and $P_{\ran(S^{\bs})^{\perp}}$ by $I - S^{\bs}\circ T^{\bs}$ (in the example of the elasticity complex, this means that $\ran(\mskw)^{\perp}=\ker(\vskw)$, and $\sym= I - \mskw\circ \vskw$). The proof of cohomology only relies on the latter forms.

In the setting of \eqref{bgg-regge}, we can define the $T^{\bs}$ operators. In fact, as the continuous level, define 
$$
T^{1}:=-\vskw, \quad T^{1}:=\mathcal{S}^{-1}, \quad T^{2}:=\frac{1}{2}\mskw. 
$$
It is obvious that $T^{1}$ maps $V_{h}^{1}$ to $W_{h}^{0}$ and $\mathcal{S}^{-1}$ maps $V_{h}^{2}$ to $W_{h}^{1}$ (as $\mathcal{S}$ is bijective).  It remains to show that $\vskw$ maps $V_{h}^{3}$ to $W_{h}^{2}$.  Let $\sigma=\sum_{e\in \mathsf{E}_{0}}\sum_{i=1}^{2}c_{e}^{i}\delta_e[{n_{i}}]\in V_{h}^{3}$ and $\varphi\in C_{c}^{\infty}(\mathbb{M})$.  Then 
$$
\langle \mskw \sigma, \varphi\rangle=\langle \sigma, \vskw\varphi\rangle=\sum_{e\in \mathsf{E}_{0}}\sum_{i=1}^{2}c_{e}^{i}\int_{e}\vskw\varphi\cdot n_{i}.
$$
Note that 
$$
\vskw \varphi\cdot n_{i}=\tr (\skw \varphi\times n_{i})=t_{e}\cdot \skw \varphi\times n_{i}\cdot t_{e}+ n_{1}\cdot \skw \varphi\times n_{i}\cdot n_{1}+ n_{2}\cdot \skw \varphi\times n_{i}\cdot n_{2}.
$$
The first term involves the tangential component of $\varphi$. For the last two terms, one cancels ($n_{i}\times n_{i}=0$) and another involves $n_{1}\times n_{2}=t_{e}$, which leads to the tangential component of $\varphi$.  This shows that $ \mskw \sigma$ is a sum of tangential delta on the edges, which is in $W_{h}^{2}$.

Now we have the necessary ingredients to run the BGG construction \cite{arnold2021complexes,vcap2022bgg} to obtain the discrete twisted complex \eqref{twisted-regge} and the discrete BGG complex \eqref{cplx:regge} (Regge sequence) from the diagram \eqref{bgg-regge}.  {The following result is a consequence.}
%\kh{We may need to introduce the T operators here.}
%From \eqref{bgg-regge}, it is obvious that we can eliminate spaces to obtain the Regge sequence.
\begin{theorem}\label{thm:cohomology-twisted}
The sequence \eqref{twisted-regge} is a complex, and its cohomology is isomorphic to that of the Regge complex \eqref{cplx:regge}.
\end{theorem}
%\begin{proof}
%The fact that \eqref{twisted-regge} is a complex follows from the above computation showing that $d^\bs$ and $S^\bs$ map to the corresponding spaces. The isomorphism of cohomology is a consequence of the BGG machinery, as we have all the ingredients (the connecting maps $\mskw$, $S$ and $\vskw$ are injective, bijective, and surjective, respectively; and the connecting maps commute with $d^\bs$ in the de~Rham complexes). See \cite{arnold2021complexes,vcap2022bgg}.
%\end{proof}

\section{Cohomology of the Regge complex}\label{sec:cohomology}

In this section, we prove the following result. 
\begin{theorem}\label{thm:main}
The  cohomology of the Regge complex  \eqref{cplx:regge} is isomorphic to the de~Rham cohomology with $\mathcal{RM}$ coefficients $\mathcal{H}^{\bs}_{dR}(\Omega)\otimes \mathcal{RM}$. 
\end{theorem}
Due to \Cref{thm:cohomology-twisted}, this also implies the cohomology of the twisted complex \eqref{twisted-regge}. 

The proof of \Cref{thm:main} follows several steps: we consider a series of auxiliary complexes and show that the cohomology of each of them is isomorphic to the $\mathcal{RM}$-valued de~Rham cohomology. The first one is \eqref{cplx:dgrm} below, starting with the space of piecewise rigid motion $C^{-1}\RM {\cong C_3(\Delta; \partial \Delta) \otimes \RM}$. The cohomology of this complex is identified with the simplicial homology. Then, we follow a series of steps to increase the regularity of \eqref{cplx:dgrm} to eventually get the Regge sequence.  More specifically, the second auxiliary sequence is \eqref{cplx:ned}, which starts from the first-type N\'ed\'elec element space $\ned$~\eqref{eq:ned} (imposing tangential continuity); the third one is \eqref{cplx:nedc}, starting with the second-type N\'ed\'elec element space $\nedc$~\eqref{eq:nedc}. With one more step of diagram chase, we derive the Regge sequence. In the rest of this section, we provide details. %\lt{The two lemmas below will be used in the following proof.}

%To prove \Cref{thm:main}, we first consider a series of auxiliary elasticity complexes, starting with the piecewise rigid motion space $C^{-1}\RM \revise{\cong C_0(\Delta; \partial \Delta) \otimes \RM}$, the first-type N\'ed\'elec element space $\ned$~\eqref{eq:ned} and the second-type N\'ed\'elec element space $\nedc$~\eqref{eq:nedc}, respectively. %Finally, we prove \Cref{thm:main}. 
%\kh{define RM, NED etc. in the notation section}\lt{Have put the definition of ned and nedc to the notation section.}

\begin{lemma}
\label{lem:integration-by-parts-inc}
    Suppose that $p = x \mapsto a + b \times x \in \mathcal{RM}$, and $A$ is a traceless matrix-valued function. Then $p \cdot \curl A \cdot n = \curl(p\cdot A)n + \frac{1}{2}n \cdot A \cdot \curl p$.
    \end{lemma}
    \begin{proof}
We have 
    \begin{equation}
        \curl(p\cdot A)\cdot n = \curl(p^l A_{li} e^i) \cdot n = \varepsilon^{ij}_{~~k} \partial_i (p^l A_{lj}) n^k = \varepsilon^{ij}_{~~k} \partial_i p^l A_{lj} n^k + \varepsilon^{ij}_{~~k} p^l \partial_i A_{lj} n^k.
    \end{equation}
    The second term $\varepsilon^{ij}_{~~k} p^l \partial_i A_{lj} n^k$ is $p\cdot \curl A \cdot n$, which is our desired term. Suppose $p = a + b \times x$, the first term is 
    $$ \varepsilon^{ij}_{~~k} \partial_i p^l A_{lj} n^k = \varepsilon^{ij}_{~~k} \partial_i (\varepsilon^{l}_{~uv} b^u x^v) A_{lj} n_k = \varepsilon^{ij}_{~~k}\varepsilon^{l}_{~uv} b^u \partial_i (x^v) A_{lj} n^k.$$ 
    Since $\partial_i x^v = \delta_i^v$, we have
    \begin{equation}
    \begin{split}
        \varepsilon^{ij}_{~~k} \partial_i p^l A_{lj} n^k & =\varepsilon^{ij}_{~~k}\varepsilon^{l}_{~ui} b^u A_{lj} n^k = (\delta_{uk}\delta^{lj} - \delta_{u}^j\delta^l_k) b^u A_{lj} n^k = -  b^j A_{lj} n^l.
    \end{split}
    \end{equation}
    The last equation comes from $A_{lj}\delta^{lj} = 0$, which is due to the traceless property of $A$. Therefore, we have
    \begin{equation}
        \curl(p \cdot A)\cdot n = p \cdot \curl A \cdot n - \frac{1}{2} n\cdot A \cdot \curl p.
    \end{equation}
    Here we use $\curl(b \times x) = 2 b$.
\end{proof}

\begin{lemma}
\label{lem:symgrad}
Let $u$ be a vector-valued function. It holds that $\grad u : (a \otimes b) - \Def u : (a \otimes b) = -  \frac{1}{2}\curl u \cdot (a \times b)$. 
\end{lemma}
\begin{proof}
The left-hand side is $\frac{1}{2} (\partial_j u_i - \partial_i u_j) a^i b^j $, and the right-hand side is $$- \frac{1}{2} \varepsilon_{~jk}^i [\curl u]_i a^j b^k = - \frac{1}{2} \varepsilon_{~jk}^i \varepsilon_{i}^{~lm} \partial_l u_m a^j b^k = -\frac{1}{2} (\delta_{j}^l \delta_{k}^m - \delta_{j}^m \delta_{k}^l) \partial_l u_m a^j b^k,$$
which is equal to the left-hand side.
\end{proof}

\subsection{The first auxiliary complex}
\label{sec:dgrm}

In this section we construct a discrete complex that starts with discontinuous $\RM$ space. The idea of introducing this complex is to provide a complex that can be related to the simiplicial relative homology, which is isomorphic to the de~Rham cohomology. To this end, we set $
    C^{-1}\RM= \{ u \in L^2(\Delta) : u|_{K} \in \RM \}
$
be the discontinuous piecewise $\RM$ space. 

For any $\sigma \in C_c^{\infty}(\Omega; \bS)$, it holds that 
\begin{equation}
    \label{eq:def}
\begin{split}
\langle \sym\grad u, \sigma \rangle & = - \langle u, \div\sigma \rangle  = -\int_{\Omega} u \cdot \div \sigma \\ 
& = -\sum_{f \in \mathsf F_0} \int_f \jump{u}_f \cdot \sigma \cdot n_f = - \sum_{f \in \mathsf F_0} \mathcal O(f,K) \int_f u|_K\cdot \sigma\cdot n_f.
\end{split}
\end{equation}
 Here $\jump{u}_f = \sum_{K} \mathcal O(f,K) u|_K$ is the jump of $u$ across the face $f$. 
This invokes us to set the following distribution for any $\mathbb M$-valued function $\sigma$:
$$\varphi_{f}^1[p] :  \sigma \mapsto -\int_f \sym \sigma :(p \otimes n_{f}).$$
Clearly, $\varphi_{f}^1[p]$ is a symmetric distribution. 
The above calculation can be summarized as:
\begin{equation}
\label{eq:def-varphi}
\Def u =  \sum_{f} \mathcal O(f,K) \varphi_{f}^1[u|_{K}].
\end{equation}

Let $\bm X^1 := \Span\{\varphi_{f}^1[p] : f \in \mathsf F_0, p \in \mathcal{RM} \}$. Then $\Def \bm X^0 \subset \bm X^1$.   
Following this idea, we now determine the remaining two distribution spaces $\bm X^2$ and $\bm X^3$. 
Set the $\bM$ valued distribution
\begin{equation}
\label{eq:varphi2}
\varphi_{e}^2[p] :  \sigma \mapsto -\int_{e} \sym \sigma : \frac{1}{2}( \curl p \otimes t_e)- \int_{e}   ({\nabla\times}(\sym\sigma) ) : (p \otimes t_e) \in \mathcal D'(\Omega; \mathbb R^{3\times 3}),
\end{equation}
and the $\mathbb V$ valued distribution
\begin{equation}
\label{eq:varphi3}
\varphi_{x}^3[p] : v\mapsto \frac{1}{2}(\curl v \cdot p)(x) +  \frac{1}{2} (v \cdot \curl p)(x)  \in 
\mathcal D'(\Omega; \mathbb R^{3}).
\end{equation}

Based on these distributions, we define $\bm X^2 = \Span\{ \varphi_e^2[p]: p \in \RM, e \in \mathsf E_0\}$ and $\bm X^3 = \Span\{ \varphi_x^3[p]: p \in \RM, x \in \mathsf V_0\}.$

\begin{lemma}
    \label{lem:varphi}
It holds that
\begin{equation}
\label{eq:inc-varphi1}
\inc \varphi_f^1[p] = \sum_{e \in \mathsf E_0} \mathcal O(e,f) \varphi_e^2[p],
\end{equation}
and
\begin{equation}
\label{eq:div-varphi2}
    \div \varphi_e^2[p] = \sum_{x \in \mathsf V_0} \mathcal O(x,e) \varphi_x^3[p]. 
\end{equation}
\end{lemma}

\begin{proof}[Proof of \Cref{lem:varphi}]

    By definition, $
    \langle \inc \sigma, \tau \rangle = \langle \sigma, \inc \tau \rangle$
    for $\tau \in C_c^{\infty}(\Omega; \bS)$.

    Set $\sigma = \varphi^1_f[p]$ for some interior face $f \in \mathsf F_0$ and $p \in \RM$. Using \Cref{lem:integration-by-parts-inc}, we have 
    \begin{equation}
        \label{eq:inc-lem1}
    \begin{split}
        \langle \inc \sigma, \tau \rangle = & - \int_{f} p \cdot {\curl (\nabla\times \tau)}\cdot n_f \\
        = & -\int_f \curl(p \cdot {\nabla\times}\tau ) \cdot n_f -\frac{1}{2} \int_{f} (\curl p)\cdot  \curl\tau \cdot n_f \\
        = & - \int_{\partial f} p\cdot ({\nabla\times}\tau )\cdot  t_{\partial f} -\frac{1}{2} \int_{\partial f} (\curl p) \cdot  \tau \cdot t_{\partial f}.
    \end{split}
    \end{equation}
    Here, the second line comes from \eqref{lem:integration-by-parts-inc} and the last line comes from the Stokes' lemma and the fact that $\curl p$ is a constant.

    Next, we calculate $\div \varphi_e^2[p]$ to obtain \eqref{eq:div-varphi2}. 
    For $v \in C_c^{\infty}(\Omega; \mathbb V)$, it follows that
    \begin{equation}
    \begin{split}
     \langle \div \varphi_e^2[p], v \rangle & = 
      \frac{1}{2} \int_{e} \sym \grad v : (\curl p\otimes t_e ) + \int_{e} {\nabla\times}(\sym \grad v) : ( p\otimes t_e).
    \end{split}
    \end{equation}
    
    By \Cref{lem:symgrad}, the first term is 
    \begin{equation}
    \begin{split}
    \int_e \sym \grad v : (\curl p \otimes t_e ) = & \int_e  \grad v : (\curl p \otimes t_e) + \frac{1}{2} \curl v \cdot (\curl p \times t_e)\\ 
     = &  \int_e \frac{\partial}{\partial t} (v \cdot \curl p) + \frac{1}{2} \curl v \cdot  (\curl p \times t_e) \\ 
     = & \int_e \frac{\partial}{\partial t} (v \cdot \curl p) +\curl v \cdot \frac{\partial}{\partial t} p,
    \end{split}
    \end{equation}
{where we have used, for $p=a+b\times x$,
    $
  \frac{1}{2}\curl p\times t_e=b\times t_e=t_e\cdot \mskw b= t_e\cdot \nabla p.
    $
    }

    The second term is 
    \begin{equation}
    \begin{split}
    \int_{e} {\nabla\times}(\sym \grad v ) : (p \otimes t_e) & = \frac{1}{2} \int_e (\nabla\times v)\nabla : (p \otimes t_e) = \frac{1}{2} \int_e \frac{\partial}{\partial t} \curl v \cdot p 
    \end{split}
    \end{equation}
    
    Combining the above two terms, we have 
    $$\langle \div \varphi_e^2[p], v \rangle =\frac{1}{2}\int_e  \frac{\partial}{\partial t} (v \cdot \curl p) + \frac{1}{2} \int_e \frac{\partial}{\partial t} (p \cdot \curl v).$$
    Therefore, we conclude \eqref{eq:div-varphi2}.
\end{proof}

\begin{proposition}
\label{prop:hom-dgrm}
The sequence 
\begin{equation}
        \label{cplx:dgrm}
    \begin{tikzcd}
   0 \ar[r] & C^{-1}\RM \ar[r,"\Def"] & \bm X^1 \ar[r,"\inc"] & \bm X^2 \ar[r,"\div"] & \bm X^3 \ar[r] & 0
    \end{tikzcd}
\end{equation}
is a complex, and the cohomology is isomorphic to $\cH^{\bs}_{dR}(\Omega)\otimes\RM  .$
\end{proposition}

\begin{proof}
The cohomology is indicated by the following diagram. 
\begin{equation}
    \label{cd:rm-kappa}
\begin{tikzcd}
0 \ar[r] & C^{-1}\RM \ar[r,"\Def"] \ar[d,"\kappa^0"] & \bm X^1 \ar[r,"\inc"] \ar[d,"\kappa^1"]& \bm X^2 \ar[r,"\div"] \ar[d,"\kappa^2"]& \bm X^3 \ar[r]\ar[d,"\kappa^3"] & 0 \\ 
0\ar[r] & \bigoplus_{K \in \mathsf K} \RM \ar[r,"\partial"] & \bigoplus_{f \in \mathsf F_0} \RM \ar[r,"\partial"] & \bigoplus_{e \in \mathsf E_0} \RM \ar[r,"\partial"] &\bigoplus_{v \in \mathsf V_0} \RM \ar[r] & 0.
\end{tikzcd}
\end{equation}

Here, we define $\kappa^0(u) = \sum_{K} K \otimes u|_K $, $\kappa^1(\varphi_f^1[p]) = f \otimes p$, $\kappa^2(\varphi_e^2[p]) = e \otimes p$ and $\kappa^3(\varphi_x^3[p]) =x \otimes p$. It follows from \eqref{eq:def-varphi} and \Cref{lem:varphi} that the diagram commutes. Since all $\kappa^k$ are isomorphisms, we conclude the result. 
\end{proof}

\subsection{The second auxiliary complex: tangential continuity}
\label{sec:ned}
Next, we impose tangential continuity to the $C^{-1}\RM$ space. The resulting space is the first-type N\'ed\'elec element space $\ned$~\eqref{eq:ned}. 
Clearly, $\ned$ is a subspace of $C^{-1}\RM$, consisting of all the functions with tangential continuity. 
The degrees of freedom are given by $u \mapsto \int_e u \cdot t_e$ for each edge $e \in \mathsf E$.

For later use, we introduce the rigid motion spaces on faces and edges. Define the projection $\Pi_f(v) = v \times n_f \times n_f = v - (v\cdot n_f) n_f$, which maps $\mathbb V$ to the tangential space of $f$.  Define the projection $\Pi_e(v) = (v \cdot t_e)t_e$, which maps $\mathbb V$ to the tangential space of $e$. 
Set $\RM(f) = \Pi_f \RM$ and $\RM(e) = \Pi_e \RM$  be the rigid motion on faces and edges. Clearly, $\dim \RM(f) = 3$ and $\dim \RM(e) = 1$.

Consider the following sequence:
\begin{equation}
    \label{cplx:rm-aux}
    \begin{tikzcd} 0 \ar[r] & \bigoplus\limits_{K \in \mathsf K} \RM \ar[r, "\partial \circ \Pi"] &  \bigoplus\limits_{f \in \mathsf F_0} \RM(f) \ar[r, "\partial \circ \Pi"] & \bigoplus\limits_{e \in \mathsf E_0} \RM(e) \ar[r] & 0 \ar[r] & 0.
    \end{tikzcd}
\end{equation}

Here we define for $u_{K}\in \mathcal{RM}$ and $u_{f}\in \mathcal{RM}(f)$, 
$$(\partial \circ \Pi) (\sum_{K \in \mathsf K} K \otimes u_K ) = \sum_{K \in \mathsf K} \sum_{f \in \mathsf F_0} \mathcal O(f,K) [f \otimes \Pi_f u_K],$$
and 
$$(\partial \circ \Pi) (\sum_{f \in \mathsf F_0} f \otimes u_f ) = \sum_{f \in \mathsf F_0} \sum_{e \in \mathsf E_0} \mathcal O(e,f) [e \otimes \Pi_e u_f]. $$

This follows from the fact $\Pi_e \circ \Pi_f =  \Pi_e$, and $\partial \circ \partial = 0.$ This complex (of the arbitrary dimension) was considered in \cite{licht2017complexes}, where the following result is proved there.

\begin{lemma}
\label{lem:hom-rm-aux}
The cohomology of \eqref{cplx:rm-aux} is 
    \begin{equation}
        \begin{tikzcd}
            0 \ar[r] & \ned \ar[r] &  0 \ar[r] &  0 \ar[r] &  0.
        \end{tikzcd}
    \end{equation}

\end{lemma}

\begin{proof}
Note that the rigid motion space (on face, on edge, resp.) is spanned by the Whitney forms $\mu_e$, which is the basis function with respect to $u\mapsto \int_e u\cdot t_e$. We distinguish the basis function in cells, faces and edges by $\mu_e^K$, $\mu_e^f$ and $\mu_e^e$, respectively. 

It then follows that $\Pi_f \mu_e^K = \mu_e^f$, and $\Pi_e \mu_e^f = \mu_e^e$.
Given $K$ and $e \subset K$, it holds that 
$(\partial \circ \Pi)  K \otimes \mu_e^K  = \sum_{f: f\subset K} \mathcal O(f,K) [f\otimes \mu_e^f] .$
Similarly, given $f$ and $e \subset f$, it holds that 
$(\partial \circ \Pi) f\otimes \mu_e^f  = \sum_{e: e\subset f} \mathcal O(e,f)[ e \otimes \mu_e^e] .$

Therefore, by reformulating \eqref{cplx:rm-aux} by $\mu_e^{\bs}$, it follows that the complex \eqref{cplx:rm-aux} is a direct sum of 
\begin{equation}
\begin{tikzcd}
    0 \ar[r] & \bigoplus\limits_{K \in \st(e)} \bR \ar[r,"\partial"] &  \bigoplus\limits_{f \in \mathsf F_0(\st(e))} \bR \ar[r,"\partial"] &  \bigoplus\limits_{e \in \mathsf E_0(\st(e))} \mathbb R \ar[r] &  0 \ar[r] & 0.
\end{tikzcd}
\end{equation}
Since $\st(e)$ is topologically trivial, the cohomology is 
\begin{equation}
    \begin{tikzcd}
        0 \ar[r] & \bR \ar[r,"\partial"] &  0 \ar[r,"\partial"] & 0 \ar[r] &  0 \ar[r] & 0.
    \end{tikzcd}
    \end{equation}
Therefore, the cohomology of \eqref{cplx:rm-aux} is $\ned$, the span of all Whitney form $\mu_e$. 
\end{proof}

We now derive the complex starting with $\ned$ by chasing the following diagram. 

\begin{equation}
    \label{cd:DGRM-Ned}
\begin{tikzcd}
0 \ar[r] & C^{-1}\RM \ar[r,"\Def"] \ar[d,"id"] & \bm X^1 \ar[r,"\inc"] \ar[d,"g^1"] & \bm X^2 \ar[d,"g^2"] \ar[r,"\div"] & \bm X^3 \ar[d] \ar[r] & 0\\
0 \ar[r] & \bigoplus\limits_{K \in \mathsf K} \RM \ar[r, "\partial\circ\Pi"] &  \bigoplus\limits_{f \in \mathsf F_0} \RM(f) \ar[r, "\partial\circ\Pi"] & \bigoplus\limits_{e \in \mathsf E_0} \RM(e) \ar[r] & 0 \ar[r] & 0.
\end{tikzcd}
\end{equation}
Here, the upper row is the $C^{-1}\RM$ complex \eqref{cplx:dgrm}, whose cohomology is $\mathcal H_{dR}^{\bs}(\Omega) \otimes \mathcal{RM}.$ 

We set
 $g^1(\varphi_f^1[p]) = f \otimes \Pi_f(p) $, and 
 $g^2(\varphi_e^2[p]) = e \otimes \Pi_e(p) $. Clearly, $g^1$ and $g^2$ are surjective. It is straightforward to verify that the diagram \eqref{cd:DGRM-Ned} commutes.

The commuting property is indicated by \Cref{prop:hom-dgrm}. 
Regarding \eqref{cd:DGRM-Ned} as a (co)chain mapping of two cochain complexes, we define the kernel of such complex. 
\begin{equation}
    \begin{tikzcd}
        0 \ar[r] & 0 \ar[r] & \ker(g^1) \ar[r,"\inc"] & \ker(g^2) \ar[r,"\div"] & \bm X^3 \ar[r] & 0
         \end{tikzcd}
        \end{equation}

We specify the kernels as follows. $\ker(g^1)$ contains the distributions $\varphi_f^1[p]$ such that $\Pi_f p = 0$. This implies that $p = v n_f$ for some $v \in \cP_1(f)$. 

Therefore, it holds that 
$$\ker(g^1) = \bm \Phi := \Span\Big\{\varphi_f^1[pn_f] : p \in \cP_1(f), f \in \mathsf F_0\Big\} \subset \bm X^1.$$
Similarly, we can define 
$$\ker(g^2) := \hat{\bm X}^2 := \Span\Big\{ \varphi_e^1[p]: p \in \RM, p\cdot t_e = 0, e \in \mathsf E_0\Big\}.$$

\begin{remark}
    In fact, the space $\bm \Phi$ can be derived from a direct calculation of $\Def u$ for $u \in \ned.$ 
By \eqref{eq:def}, it holds that $
\Def u = \sum_{f \in \mathsf F_0} \varphi_f^1\Big[\jump{u}_f \Big].$ 
Since $u \cdot n_f$ is continuous, the jump term $\jump{u}_f = \jump{u \cdot n_f}_f n_f$ is a purely normal vector term, whose value is in $\cP_1(f) n_f$. Therefore, $\Def \ned \subset \bm \Phi$. 
\end{remark}

The following result follows from the above observation and \Cref{lem:hom-rm-aux}.

\begin{proposition}
\label{prop:ned}
The sequence 
\begin{equation}
\label{cplx:ned}
    \begin{tikzcd}
        0 \ar[r] & \ned \ar[r, "\Def"] & \bm \Phi \ar[r,"\inc"] & \hat{\bm X}^2 \ar[r,"\div"] & \bm X^3 \ar[r] & 0
         \end{tikzcd}
        \end{equation}
        is a complex. The cohomology of \eqref{cplx:ned} is $\mathcal H^{\bs}_{dR}(\Omega)\otimes \RM $.
\end{proposition}

\begin{proof}
Denote by $\mathscr C$ the complex \eqref{cplx:ned}. 
Note that the sequence of cochain complexes:
\begin{equation}
\begin{tikzcd}
0 \ar[r] & 0 \ar[r, "\Def"] \ar[d] & \bm \Phi \ar[r,"\inc"] \ar[d] & \hat{\bm X}^2 \ar[r,"\div"] \ar[d]& \bm X^3 \ar[r] \ar[d] & 0 \\
0 \ar[r] & C^{-1}\RM \ar[r,"\Def"] \ar[d,"id"] & \bm X^1 \ar[r,"\inc"] \ar[d,"g^1"] & \bm X^2 \ar[d,"g^2"] \ar[r,"\div"] & \bm X^3 \ar[d] \ar[r] & 0\\
0 \ar[r] & \bigoplus\limits_{K \in \mathsf K} \RM \ar[r, "\partial\circ\Pi"] &  \bigoplus\limits_{f \in \mathsf F_0} \RM(f) \ar[r, "\partial\circ\Pi"] & \bigoplus\limits_{e \in \mathsf E_0} \RM(e) \ar[r] & 0 \ar[r] & 0
\end{tikzcd}
\end{equation}
is exact, which leads to the following long exact sequence:
\begin{equation}
    \begin{tikzcd}[column sep=small]
      0 \arrow{r} & 0 \arrow{r} \arrow{d} &\ker(\inc) \cap \bm \Phi \arrow{r}\arrow{d} & \mathcal H^2(\mathscr C) \arrow{r} \arrow{d} & \mathcal H^3(\mathscr C) \arrow{r} \arrow{d} & 0 \\
   0 \arrow{r} &  \mathcal H^0_{dR}(\Omega)\otimes \RM \arrow[r, ""{coordinate, name=Z}] \arrow{d} &  \mathcal  H^1_{dR}(\Omega)\otimes \mathcal{RM{ }} \arrow[r, ""{coordinate, name=Y}] \arrow{d} &  \mathcal H^2_{dR}(\Omega) \otimes \mathcal{RM} \arrow[r, ""{coordinate, name=X}] \arrow{d}&  \mathcal H^3_{dR}(\Omega) \otimes \mathcal{RM} \ar[r] \ar[d]&  0\\
    0 \arrow{r}&\ned \arrow{r} \arrow[uur, rounded corners, dashed, to path={ -- ([yshift=-4ex]\tikztostart.south)
    -| (Z) [near end]\tikztonodes
    |- ([yshift=4ex]\tikztotarget.north)
    -- (\tikztotarget)}]
    & 0 \arrow{r} \arrow[uur, rounded corners, dashed, to path={ -- ([yshift=-4ex]\tikztostart.south)
    -| (Y) [near end]\tikztonodes
    |- ([yshift=4ex]\tikztotarget.north)
    -- (\tikztotarget)}] &0  \arrow{r} \arrow[uur, rounded corners, dashed, to path={ -- ([yshift=-4ex]\tikztostart.south)
    -| (X) [near end]\tikztonodes
    |- ([yshift=4ex]\tikztotarget.north)
    -- (\tikztotarget)}] & 0 \ar[r] & 0.
     \end{tikzcd}
    \end{equation}
It can be argued that the linking map from $\ned$ to $\ker(\inc) \cap \bm \Phi$ is $\Def$. It follows that $\mathcal  H^2(\mathscr C) \cong \mathcal H^2_{dR}(\Omega) \otimes \RM$ and $\mathcal H^3(\mathscr C) \cong \mathcal H^3_{dR}(\Omega) \otimes \RM$. 
The long exact sequence reads as 
\begin{equation}
\label{cplx:long-exact}
\begin{tikzcd}[column sep = small]
 0 \ar[r] & \cH^0_{dR}(\Omega) \otimes\RM \ar[r] & \ned \ar[r,"\Def"] & \ker(\inc) \cap \bm \Phi \ar[r,"j"] &  \cH^1_{dR}(\Omega)\otimes\RM \ar[r] & 0.
\end{tikzcd}
\end{equation}

Observe that 
$ \mathcal H^0(\mathscr C) :=\ker(\Def) \cap \ned$ and 
$\displaystyle \mathcal H^1(\mathscr C) := {[\ker(\inc)\cap \bm \Phi]}/{\Def \ned}.$
The exactness of \eqref{cplx:long-exact} tells us that 
$$\ker(\Def) \cap \ned =  \cH^0_{dR}(\Omega)\otimes\RM ,$$
and 
$$[{\ker(\inc) \cap \bm \Phi}]/{\Def \ned} =  [{\ker(\inc) \cap \bm \Phi}]/{\ker(j)} \cong \ran(j) = \cH^1_{dR}(\Omega) \otimes \RM.$$
This concludes the result. 

\end{proof}
\begin{remark}
A similar idea is adopted in \cite{hu2023distributional} to calculate the cohomology of the distributional Hessian complexes in 2D and 3D. 
\end{remark}

\subsection{The third auxiliary complex: enriching the local shape function space}
\label{sec:nedc}
In this subsection, we construct a complex that starts with the space $\nedc$~\eqref{eq:nedc}. 
The introduction of the second type N\'ed\'elec element is motivated by the following observations. First, the shape function is $\mathcal P_1$, which is the same as $\lag$. Second, the two spaces $\nedc$ and $\ned$ both have the tangential continuity, which implies that additional distribution terms would not appear. 
% The space is defined as 
% \begin{equation}
% \nedc : = \{u: u|_K \in \cP_1 \otimes \mathbb V \text{ for each cell } K,  u \times n_f \text{ is continuous on each face }f\}.
% \end{equation}
% The classical degrees of freedom are 
% $u \mapsto \int_{e} (u \cdot t_e)p,$ for all $p \in \mathcal P_1(e).$
Here we introduce a modified set of degrees of freedom of $\nedc$.
\begin{lemma}
\label{lem:nedc-uni}
For $u \in\cP_1 \otimes \mathbb V$, define the following degrees of freedom:
\begin{enumerate}
\item[(1)] $\displaystyle \int_e (u\cdot t_e)$ for each edge $e$; 
\item[(2)] $\displaystyle \int_e (\frac{\partial}{\partial t_e}u\cdot t_e)$ for each edge $e$.
\end{enumerate}
Then, the degrees of freedom are unisolvent. The resulting finite element space are $\nedc$. 
\end{lemma}
\begin{proof}
We first show the unisolvency. The number of degrees of freedom is $12=3 \dim \mathcal P_1$. Next, suppose that $u \in \mathcal P_1 \otimes \mathbb V$ vanishes for those degrees of freedom. Then, we have $\frac{\partial}{\partial t} u\cdot t_e = 0$ on each edge. This implies that $u\cdot t_e$ is a constant along the edge. Further, the first degrees of freedom indicates that $u \cdot t_e = 0 $ on each edge. 

As a result, we can conclude that $u = 0$, which proves the unisolvency. The proof also shows that $u\cdot t_e$ is single-valued, which implies that the resulting finite element space is $\nedc$.
\end{proof}

Locally, we have the following exact sequence
$$\begin{tikzcd} \RM \ar[r] & \mathcal P_1 \otimes \mathbb V \ar[r,"\Def"] & \bS \ar[r] & 0.
\end{tikzcd}$$
For the global version, the piecewise $\deff$ operator maps to the Regge space. 

\begin{lemma}
\label{lem:def-nedc}
It holds that $ \Def_h\nedc \subset \reg$. Here $\Def_h$ is the piecewise deformation. Consequently, $ \Def \nedc \subset \reg \oplus \bm \Phi$.
\end{lemma}
\begin{proof}[Proof of \Cref{lem:def-nedc}]
    By the definition of $\nedc$, for $v \in \nedc$, the moment
    $\int_{e} \frac{\partial}{\partial t_e} (v\cdot t_e)$
    is single-valued for each edge $e \in \mathsf E$.
    Taking $\sigma = \Def_h v$ yields that 
    $\int_{e} t_e \cdot \sigma \cdot t_e$ 
    is single-valued for each edge $e$. It follows from the definition of $\reg$ that $\Def_h v \in \reg.$
    
    Now we calculate $\Def v$ by definition. 
    It holds that for $\tau \in C_c^{\infty}(\Omega; \bS)$, 
    \begin{equation}
    \begin{split}
        \langle \Def w, \tau \rangle & = -\langle w, \div \tau \rangle 
     = - \int_{\Omega} w \cdot \div \tau \\ 
    &  = \int_{\Omega} \Def_h w : \tau - \sum_{f \in \mathsf F_0} \int_f \jump{w}_f \cdot (\tau n_f) \\ 
    & = \int_{\Omega} \Def_h w : \tau - \sum_{f \in \mathsf F_0} \int_f \jump{w\cdot n_f}_f \cdot (n_f\cdot \tau \cdot n_f).
    \end{split}
    \end{equation}
    Here, the last line comes from the tangential contniuity of $w \in \nedc$. 
    Note that the first term lies in $\reg$ and the second term lies in $\bm \Phi$. Therefore, we prove that $\Def \nedc \subset \reg \oplus \bm \Phi$.
    \end{proof}

Moreover, we have the following result.
\begin{proposition}
    \label{prop:nedc}
    The sequence
    \begin{equation}
        \label{cplx:nedc}
    \begin{tikzcd}
   0 \ar[r] & \nedc \ar[r,"\Def"] & \reg \oplus \bm \Phi \ar[r,"\inc"] & \hat{\bm X}^2 \ar[r,"\div"] & \bm X^3 \ar[r] & 0
    \end{tikzcd}
\end{equation}
is a complex. The cohomology is isomorphic to $ \cH^{\bs}_{dR}(\Omega)\otimes\RM $.
\end{proposition}

\begin{proof}[Proof of \Cref{prop:nedc}: Complex part] 
To show \eqref{cplx:nedc} 
is a complex, it remains show that $\inc \reg \subset \hat{\bm{X}}^2$. By the fact that the Regge sequence is a complex, it holds that $\inc \reg\subset {(\reg_0)'}$. It suffices to show 
$(\reg_0)' \subset \hat{\bm X}^2$.

Choose any point $x_0 \in e$, and set $p = t_e \times  x - t_e \times x_0$. Therefore, $p = 0$ on $e$, and $\curl p = 2t_e$. Substituting these into \eqref{eq:varphi2}, we obtain that 
\begin{equation}\label{eq:delta_tt}
\varphi_e^2[p] = \sigma \mapsto -\int_e \sym \sigma :(t_e \otimes t_e) =  -\delta_{e}[t_e t_e^T].\end{equation}
Therefore, we prove that $(\reg_0)' \subset {\hat{\bm{X}}^2}$, which implies the result. 
\end{proof}

Now we consider the cohomology of \eqref{cplx:nedc}. The key is that for functions $\sigma \in \reg$, we can always find $w \in \nedc$ such that $\Def_h w = \sigma$. 
In fact, by the degrees of freedom, we construct $w \in \nedc$ such that 
$$\int_e (w\cdot t_e) = 0,
\text{  and  } 
\int_e (\frac{\partial}{\partial t_e} w \cdot t_e) = \int_{e} t_e \cdot \sigma \cdot t_e.$$

Therefore, it holds that $\Def_h w = \sigma$. Denote by $w := \mathcal L\sigma$ and define $$\cK\sigma := \Def \mathcal L\sigma - \sigma \in \bm \Phi.$$ 

\begin{lemma}
\label{lem:K}
For $\sigma \in \reg$ and $\psi \in \bm \Phi$, the following results hold:
    \begin{enumerate}
    \item $\inc \cK \sigma = -\inc \sigma$;
    \item $\inc(\sigma + \psi) = 0$ if and only if $\inc(\psi - \cK \sigma) = 0$;
    \item $\sigma + \psi \in \Def \nedc$ if and only if $\psi - \cK \sigma \in \Def \ned.$
    \end{enumerate}
\end{lemma}
\begin{proof}
The first part comes from $\inc \cK \sigma = \inc \Def \cL \sigma - \inc \sigma = - \inc \sigma.$ Obviously, (1) implies (2).

Now we prove (3). Since $\sigma + \cK \sigma \in \Def \nedc$, it holds that 
$$\sigma + \psi \in \Def \nedc \Longleftrightarrow \psi - \cK \sigma \in \deff \nedc.$$
 Suppose that $\Def v = \psi - \cK \sigma$. Note that $\psi - \cK \sigma$ is a purely distribution term. This implies that the piecewise symmetric gradient of $v$ is zero. Therefore, it holds that $v \in \ned.$ This can conclude the result. 
\end{proof}

Now we are ready to prove the remaining part of \Cref{prop:nedc}.
\begin{proof}[Proof of \Cref{prop:nedc}: Cohomology part]
It follows from a case-by-case argument. By checking the piecewise deformation, it holds that $\ker(\Def) \cap \ned = \ker(\Def) \cap \nedc$. It suffices to show that the mapping $\psi \mapsto (0, \psi)$, with inverse $(\sigma, \psi) \mapsto \psi - \mathcal K \sigma$, is an isomorphism between $ [{\ker(\inc) \cap \bm \Phi}]/{\Def \ned}$ and $[{\ker(\inc) \cap (\reg \oplus \bm \Phi)}]/{\Def \ned^c}$. It suffices to show that the two mappings are well-defined. The mapping $\psi \mapsto (0, \psi)$ is straightforward, and the inverse $(\sigma, \psi) \mapsto \psi - \mathcal K \sigma$ is ensured by \Cref{lem:K}.  
\end{proof}

\subsection{Final proof}
\label{sec:final}

We aim to introduce the normal continuity to the space $\nedc$, obtaining the desired Lagrange space $\lag$. The proof is similar to that of \Cref{sec:ned}, but the auxiliary complex considered in this subsection is more complicated. The proof of the cohomology of the auxiliary complex is postponed to \Cref{sec:proof-p1n}.

Note that the distribution in $\bm \Phi$ can be identified as elements in the free abelian group $\bigoplus_{f \in \mathsf F_0} \mathcal P_1^n(f)$, with $\mathcal P_1^n(f) :=  \mathcal P_1(f) n_f$, by the following mapping 
\begin{equation}
j^1(\varphi_f^1[p n_{f}]) = pn_{f}.
\end{equation}

Define 
$$\mathcal P_1^n(e) := \Span\{ \mathcal P_1(e) n_1^e, \mathcal P_1(e) n_2^e \} \text{  and  }\mathcal P_1^n(x) := \Span\{ \mathcal P_1(x) \otimes \mathbb V \},$$ respectively. 
{Note that $\mathcal P_1(x)$ actually is a real number.} As a result, we have $\dim \mathcal P_1^n(f) = 3$, $\dim \mathcal P_1^n(e) = 4$ and $\dim \mathcal P_1^n(x) = 3$. 

The goal is to make the following diagram commutes and obtain the cohomology of the Regge complex by diagram chasing argument like \Cref{sec:ned}.
\begin{equation}
    \label{cd:Nedc-lag}
\begin{tikzcd}
0 \ar[r] & \nedc \ar[r,"\Def"] \ar[d,"id"] & \reg \oplus \bm \Phi \ar[r,"\inc"] \ar[d,"j^1"] & \hat{\bm X}^2 \ar[d,"j^2"] \ar[r,"\div"] & \bm X^3 \ar[d, "j^3"] \ar[r] & 0\\
0 \ar[r] & \nedc \ar[r,"\widetilde{\partial}"] & \bigoplus\limits_{f \in \mathsf F_0} \mathcal P_1^n(f) \ar[r,"\widetilde{\partial}"] & \bigoplus\limits_{e \in \mathsf E_0} \mathcal P_1^n(e) \ar[r,"\widetilde{\partial}"] & \bigoplus\limits_{x \in \mathsf V_0} \mathcal P_1^n(x) \ar[r] & 0.
\end{tikzcd}
\end{equation}

By \Cref{prop:nedc}, the upper row is the complex \eqref{cplx:nedc}, whose cohomology is $\mathcal H_{dR}^{\bs}(\Omega) \otimes \mathcal{RM}.$
 The lower row is defined as follows: 
 %\kh{no need to write this complex again?}\lt{Referred several times.}
\begin{equation}
\label{cplx:p1n-aux}
\begin{tikzcd}
0 \ar[r] & \nedc \ar[r,"\widetilde{\partial}"] & \bigoplus\limits_{f \in \mathsf F_0} \mathcal P_1^n(f) \ar[r,"\widetilde{\partial}"] & \bigoplus\limits_{e \in \mathsf E_0} \mathcal P_1^n(e) \ar[r,"\widetilde{\partial}"] & \bigoplus\limits_{x \in \mathsf V_0} \mathcal P_1^n(x) \ar[r] & 0.
\end{tikzcd}
\end{equation}
% \kh{mention somewhere that $P_1^n(x)$ is a number.}
We define 
$$\widetilde{\partial} u = \sum_{K \in \mathsf K}\sum_{f \in \mathsf F_0} \mathcal O(f, K) [f \otimes u|_K] .$$ 
%\lt{Actually, it is $\jump{u}_f$, which by continuity lies in $P_1^n$.}
$$\widetilde{\partial} (f \otimes p_f ) = \sum_{e \subset f} \mathcal O(e,f) [e \otimes (p_f)|_e ],\text{  and  }\widetilde{\partial} (e \otimes p_e ) = \sum_{ x \in  e} \mathcal O(x,e) [x\otimes (p_e)|_x].$$
It follows that $[\mathcal{P}_1^n(f)]|_e \subset \mathcal{P}^n_1(e)$ whenever $e \subset f$, which indicates that \eqref{cplx:p1n-aux}, i.e., the lower row of \eqref{cd:Nedc-lag} is a complex. 

The following lemma shows the cohomology of this complex. We postpone the proof to \Cref{sec:proof-p1n}. 
\begin{lemma}
\label{lem:hom-p1n}
The sequence \eqref{cplx:p1n-aux} is a complex. Its cohomology is 
    \begin{equation}
        \begin{tikzcd}
        0 \ar[r] & \lag \ar[r] & 0 \ar[r] & 0 \ar[r] & 0 \ar[r] & 0.
        \end{tikzcd}
    \end{equation}
\end{lemma}

The vertical mappings of \eqref{cd:Nedc-lag} are specified as follows. Set 
\begin{itemize}
\item[-] $j^1(\reg) = 0$, and $j^1(\varphi_f^1[pn_f]) = f \otimes (pn_f) $ for $p \in \mathcal P_1(f)$.
\item[-] $j^2(\varphi_e^2[p]) = e \otimes p|_e $. It follows from $\varphi_e^2[p] \in \hat{\bm X}^2$ if and only if $p\cdot t_e = 0$ that $j^2$ maps $\hat{\bm X}^2$ to $\bigoplus_{e\in\mathsf E_0} \mathcal P_1^n(e) $.
\item[-] $j^3(\varphi_x^3[p]) = p(x).$  
\end{itemize}
It is easy to see that each $j^k$ is surjective. 

\begin{proposition}
    Under the above setup, the diagram \eqref{cd:Nedc-lag} commutes. Moreover,
the kernel of the vertical maps in \eqref{cd:Nedc-lag} is
\begin{equation}
\begin{tikzcd}
0 \ar[r] & 0 \ar[r] & \reg \ar[r] & (\reg_0)'\ar[r] & (\lag_0)' \ar[r] & 0,
\end{tikzcd}
\end{equation}
the tail of the Regge complex \eqref{cplx:regge}.
\end{proposition}
\begin{proof}
    We first prove the commuting property.
\begin{itemize}
    \item[-]
The identity $\widetilde{\partial} = j^1(\Def)$ comes from \eqref{eq:def}. 
\item[-] Since $\inc(\reg)$ only gives the term with form $\delta_{e}[t_e t_e^T]$, which is in the kernel of $j^2$. To prove $j^2 \circ \inc = \widetilde{\partial} \circ j^1$, it suffices to consider the term in $\bm \Phi$. For $\varphi^1_f[pn_f] \in \bm \Phi$. We have 
$
\inc \varphi^1_f[p n_f] = \sum_{e \subset f} \mathcal O(e,f) \varphi^2_e[p n_f].
$
As a result, 
$$j^2 (\inc \varphi^1_f[p n_f]) = \sum_{e \subset f} \mathcal O(e,f) [e \otimes (pn_f)|_e ] = \widetilde{\partial} j^1 (\inc \varphi_f[p n_f])$$
\item[-] Similarly, we can prove $\widetilde{\partial} j^2 = j^3 \div.$
\end{itemize}
Clearly, the kernel of $j^1$ is $\reg$. The kernel of $j^2$ is spanned by those $\varphi_e^2[p]$ such that $p|_e = 0$. This implies that $p = t \times x - t \times x_0$ for some $x_0 \in e$. For such $p$, it follows from \eqref{eq:delta_tt} that 
$
\varphi_e^2[p] = -2\delta_e[{t_e t_e^T}].
$
%\kh{``2'' is typo? should be $-\delta_e[{t_e t_e^T}] $?}\lt{$-2$, I think.}\kh{seems not consistent with \eqref{eq:delta_tt}.}
Thus, $\ker(j^2) = (\reg_0)'$.

Similarly, $\ker(j^3)$ is spanned by those $\varphi_x^3[p]$ with $p(y) = 0$ for all $y \in \mathsf V_0$. Such $p$ has the form $p = b \times (x - y)$ at vertex $y$. Substituting this form into \eqref{eq:varphi3}, it follows that $\varphi_x^3[p] = 2\delta_x[b]$. Therefore, $\ker(j^3) = (\lag_0)'$.
\end{proof}

\begin{proof}[Proof of \Cref{thm:main}]
The proof follows a similar argument of \Cref{prop:ned}. Denote by $\mathscr R$ the Regge complex. Note that the sequence of cochain complexes:
\begin{equation}
\begin{tikzcd}[column sep=small]
0 \ar[r] & 0 \ar[r, "\Def"] \ar[d] & \reg \ar[r,"\inc"] \ar[d] & (\reg_0)' \ar[r,"\div"] \ar[d]& (\lag_0)' \ar[r] \ar[d] & 0 \\
0 \ar[r] & \nedc \ar[r,"\Def"] \ar[d,"id"] & \reg \oplus \bm \Phi \ar[r,"\inc"] \ar[d,"j^1"] & \hat{\bm X}^2 \ar[d,"j^2"] \ar[r,"\div"] & \bm X^3 \ar[d] \ar[r] & 0\\
0 \ar[r] & \nedc \ar[r,"\widetilde{\partial}"] & \bigoplus\limits_{f \in \mathsf F_0} \mathcal P_1^n(f) \ar[r,"\widetilde{\partial}"] & \bigoplus\limits_{e \in \mathsf E_0} \mathcal P_1^n(e) \ar[r,"\widetilde{\partial}"] & \bigoplus\limits_{x \in \mathsf V_0} \mathcal P_1^n(x) \ar[r] & 0.
\end{tikzcd}
\end{equation}
is exact, which leads to the following long exact sequence:
\begin{equation}
    \begin{tikzcd}[column sep=tiny]
      0 \arrow{r} & 0 \arrow{r} \arrow{d} &\ker(\inc) \cap \reg \arrow{r}\arrow{d} & \mathcal H^2(\mathscr R) \arrow{r} \arrow{d} & \mathcal H^3(\mathscr R) \arrow{r} \arrow{d} & 0 \\
   0 \arrow{r} &  \mathcal H^0_{dR}(\Omega)\otimes \RM \arrow[r, ""{coordinate, name=Z}] \arrow{d} &  \mathcal  H^1_{dR}(\Omega)\otimes \mathcal{RM{ }} \arrow[r, ""{coordinate, name=Y}] \arrow{d} &  \mathcal H^2_{dR}(\Omega) \otimes \mathcal{RM} \arrow[r, ""{coordinate, name=X}] \arrow{d}&  \mathcal H^3_{dR}(\Omega) \otimes \mathcal{RM} \ar[r] \ar[d]&  0\\
    0 \arrow{r}&\lag \arrow{r} \arrow[uur, rounded corners, dashed, to path={ -- ([yshift=-4ex]\tikztostart.south)
    -| (Z) [near end]\tikztonodes
    |- ([yshift=4ex]\tikztotarget.north)
    -- (\tikztotarget)}]
    & 0 \arrow{r} \arrow[uur, rounded corners, dashed, to path={ -- ([yshift=-4ex]\tikztostart.south)
    -| (Y) [near end]\tikztonodes
    |- ([yshift=4ex]\tikztotarget.north)
    -- (\tikztotarget)}] &0  \arrow{r} \arrow[uur, rounded corners, dashed, to path={ -- ([yshift=-4ex]\tikztostart.south)
    -| (X) [near end]\tikztonodes
    |- ([yshift=4ex]\tikztotarget.north)
    -- (\tikztotarget)}] & 0 \ar[r] & 0.
     \end{tikzcd}
    \end{equation}
    Similar to the proof of \Cref{prop:ned}, we can conclude that the cohomology of $\mathscr R$ is isomorphic to $\cH^{\bs}_{dR}(\Omega)\otimes\RM$.
\end{proof}

\subsection{Proof of \Cref{lem:hom-p1n}}
\label{sec:proof-p1n}
To complete the proof, we now prove \Cref{lem:hom-p1n}. We start by recalling some results from the distribution Hessian complex in \cite{hu2023distributional}. 
The complex \eqref{cplx:p1n-aux} resembles the construction of the distributional Hessian complex in \cite{hu2023distributional}.

Recall the distribution Hessian complex in three dimensions. It has the following form:
\begin{equation}
\label{cplx:hessian}
\begin{tikzcd}
0 \ar[r] & V^0 \ar[r,"\hess"] & V^1 \ar[r,"\curl"] & V^2 \ar[r,"\div"] & V^3 \ar[r] & 0. 
\end{tikzcd}
\end{equation}
Here $v^0 = \mathrm{Lag}$, the scalar Lagrange $\mathcal P_1$ function, 
$$V^1 = \Span\{\delta_f[n_fn_f^T]: f \in \mathsf F_0\},$$
$$V^2 = \Span\{ \delta_e[n_{1}^et_{e}^T],  \delta_e[n_{2}^et_{e}^T] : e \in \mathsf E_0 \},
\text{  and  }V^3 = \Span \{ \delta_{x}[v]: x \in \mathsf V_0, v \in \mathbb V\}.$$
The following geometric interpretation, as an isomorphism of the Hessian complex and the sequence of $\mathcal P_0^n$'s, was observed in \cite{hu2023distributional}. 
\begin{equation}
\label{cd:geometric}
\begin{tikzcd}
0 \ar[r]  & V^1 \ar[r,"\curl"] \ar[d] & V^2 \ar[r,"\div"] \ar[d] & V^3 \ar[r] \ar[d] & 0 \\ 
0 \ar[r] & \bigoplus\limits_{f \in \mathsf F_0} \mathcal P_0^n(f) \ar[r,"\widetilde{\partial}"] & \bigoplus\limits_{e \in \mathsf E_0} \mathcal P_0^n(e) \ar[r,"\widetilde{\partial}"] & \bigoplus\limits_{x \in \mathsf V_0} \mathcal P_0^n(x) \ar[r] & 0.
\end{tikzcd}
\end{equation}
Here, the space $\mathcal P_0^n(f)$ is the vector space that normal to $f$, $\mathcal P_0^n(e)$ is the space that normal to $e$, and $\mathcal P_0^n(x) = \mathbb V$. 

By the above geometric interpretation and the cohomology of distribution Hessian complex, we can derive the following result.
\begin{lemma}
    \label{lem:hom-p0n}
The sequence 
    \begin{equation}
        \label{eq:p0n}
        \begin{tikzcd}
            0 \ar[r] & \mathrm{Lag} \ar[r,"\widetilde{\hess}"] & \bigoplus\limits_{f \in \mathsf F_0} \mathcal{P}_0^n(f) \ar[r,"\widetilde{\partial}"] & \bigoplus\limits_{e \in \mathsf E_0} \mathcal{P}_0^n(e) \ar[r,"\widetilde{\partial}"] & \bigoplus\limits_{x \in \mathsf V_0} \mathcal{P}_0^n(x) \ar[r] & 0
            \end{tikzcd}
        \end{equation} is a complex, and the cohomology is isomorphic to $  \cH^{\bs}_{dR}(\Omega)\otimes\cP_1$. Here $\widetilde{\hess}$ maps $u \in \mathrm{Lag}$ to $\sum_{f \in \mathsf F_0} f \otimes \jump{\frac{\partial u}{\partial n_f}}.$
        
        As a consequence, when $\Omega$ is homologically trivial, the cohomology of \eqref{eq:p0n} vanishes but leaves $\mathcal P_1$ at index 0. 
\end{lemma}

Now we can give the proof of \Cref{lem:hom-p1n}, based on the above connection and the geometric decomposition of $\cP_1$.

\begin{proof}[Proof of \Cref{lem:hom-p1n}]

Clearly, the kernel at $\nedc$ is $\lag$. It remains to prove that the other cohomology groups are trivial. 

We first consider a modification of \eqref{cplx:p1n-aux}:
\begin{equation}
    \label{cplx:p1n-aux-2}
    \begin{tikzcd}
    0 \ar[r] & 0 \ar[r,"\widetilde{\partial}"] & \bigoplus_{f \in \mathsf F_0} \mathcal{P}_1^n(f) \ar[r,"\widetilde{\partial}"] & \bigoplus_{e \in \mathsf E_0} \mathcal{P}_1^n(e) \ar[r,"\widetilde{\partial}"] & \bigoplus_{x \in \mathsf V_0} \mathcal{P}_1^n(x) \ar[r] & 0.
    \end{tikzcd}
    \end{equation}

For each $v$, we define $\st(v)$ to be the local patch of $v$. It follows from the geometric decomposition of the Lagrange space that the complex \eqref{cplx:p1n-aux-2} can be identified as a direct sum of 
\begin{equation}
    \label{cplx:p0n-st}
    \begin{tikzcd}[column sep=small]
        0 \ar[r] & 0 \ar[r,"\widetilde{\partial}"] & \bigoplus\limits_{f \in \mathsf F_0(\st(v))} \mathcal{P}_0^n(f) \ar[r,"\widetilde{\partial}"] & \bigoplus\limits_{e \in \mathsf E_0(\st(v))} \mathcal{P}_0^n(e) \ar[r,"\widetilde{\partial}"] & \bigoplus\limits_{x \in \mathsf V_0(\st(v))} \mathcal{P}_0^n(x) \ar[r] & 0,
        \end{tikzcd}
    \end{equation}
for all vertices $x \in \mathsf V$.

 For each vertex $x$, denote by $\lambda_x$ the barycentric coordinate function with respect to $x$. Given a face $f$ and vertex $x \in f$, the term $\lambda_x f \otimes n_f$ is mapped to $\sum_{e: x \in e \subset f} \mathcal O(e,f) \lambda_x  e \otimes  n_f $ by $\widetilde{\partial}.$ Similar equations can be derived for a given edge $e$ and $x \in e$. Since $\lambda_x$ forms a basis of $\cP_1(f)$ ($\cP_1(e), \cP_1(x)$, resp.), the complex can be decomposed to a direct sum of \eqref{cplx:p0n-st}.

Since $\st(v)$ is topologically trivial, it then follows from \eqref{lem:hom-p0n} that the cohomology of \eqref{cplx:p1n-aux-2} is a direct sum of 
\begin{equation}
\label{cplx:hom-p0n-st}
\begin{tikzcd}
0 \ar[r] & 0 \ar[r] & \widetilde{\hess} \Big(\textrm{Lag}(\st(v))\Big)\ar[r] & 0 \ar[r] & 0 \ar[r] & 0.
\end{tikzcd}
\end{equation}
Now we consider the original complex \eqref{cplx:p1n-aux}. For each $v \in \mathsf V$, consider $w_v \in \textrm{Lag}(\st(v))$. Set $w = \sum_{v \in \mathsf V} \nabla w_v \lambda_v \in \nedc$. Therefore, $\widetilde{\partial} {\nabla w} =  \sum_{f \in \mathsf F_0} f \otimes \jump{\nabla w}  = \sum_{v\in\mathsf V} \lambda_v \widetilde{\hess} w_v$. Combining with \eqref{cplx:hom-p0n-st}, it follows that \eqref{cplx:p1n-aux} is exact at index 1. The cohomology at index 0 is straightforward. Therefore, we conclude the result. 

\end{proof}

\section*{Acknowledgement}

The work of KH was supported by a Royal Society University Research Fellowship (URF$\backslash$R1$\backslash$221398).

The authors would like to thank Douglas N. Arnold for helpful discussions.

% Similar result has been observed and discussed in [Licht]. For example 
% \begin{equation}
%     \label{cplx:p0n-aux}
%     \begin{tikzcd}
%         0 \ar[r] & \nabla \lag \ar[r,"\widetilde{\partial}"] & \bigoplus_{f \in \mathsf F_0} P_0^n(f) \ar[r,"\widetilde{\partial}"] & \bigoplus_{e \in \mathsf E_0} P_0^n(e) \ar[r,"\widetilde{\partial}"] & \bigoplus_{x \in \mathsf V_0} P_0^n(x) \ar[r] & 0.
%         \end{tikzcd}
%     \end{equation}
\bibliographystyle{siam}      % mathematics and physical sciences
\bibliography{ref}{}   % name your BibTeX data base

\end{document}